\documentclass[]{amsart}

\usepackage{bbding}
\usepackage[dvips]{graphicx}
\usepackage{amsfonts,amssymb,amsmath,amsthm,amscd}
\allowdisplaybreaks
\usepackage{caption}
\usepackage{subcaption}
\usepackage{mathrsfs}
\usepackage{xcolor}
\usepackage{empheq}
\usepackage{adforn}
\usepackage{cancel}
\usepackage{xr}
\usepackage{mdframed}
\usepackage{tikz-cd}
\usepackage{tikz}
\usepackage{amsfonts}
\usepackage{mathdots}
\usepackage{pinlabel}

\usepackage{enumerate}
\usepackage{lmodern}
\usepackage{mdframed}
\usepackage{stmaryrd}
\usepackage{blindtext}
\usepackage{leftidx}
\usepackage{easytable}
\usepackage{enumitem}
\usepackage{marginnote} % for better margin notes
\usepackage{easyReview}

\usepackage{tgpagella}
\usepackage[T1]{fontenc}

\usepackage[marginparwidth=2.5cm]{geometry}
\definecolor{navie}{RGB}{0, 70, 140}

\usepackage[backref=page]{hyperref}
\usepackage{xcolor}
\hypersetup{
	colorlinks,
	linkcolor={red!60!black},
	citecolor={navie},
	urlcolor={navie}
}
\PassOptionsToPackage{linktocpage}{hyperref}

\makeatletter
\newtheorem*{rep@theorem}{\rep@title}
\newcommand{\newreptheorem}[2]{%
	\newenvironment{rep#1}[1]{%
		\def\rep@title{#2 \ref{##1}}%
		\begin{rep@theorem}}%
		{\end{rep@theorem}}}
\makeatother

\makeatletter
\newtheorem*{rep@cor}{\rep@title}
\newcommand{\newrepcor}[2]{%
	\newenvironment{rep#1}[1]{%
		\def\rep@title{#2 \ref{##1}}%
		\begin{rep@cor}}%
		{\end{rep@cor}}}
\makeatother

\makeatletter
\newtheorem*{rep@prop}{\rep@title}
\newcommand{\newrepprop}[2]{%
	\newenvironment{rep#1}[1]{%
		\def\rep@title{#2 \ref{##1}}%
		\begin{rep@prop}}%
		{\end{rep@prop}}}
\makeatother
\newtheorem{corollary}{Corollary}[section]
\newtheorem{corx}{Corollary}

\newtheorem{thmx}[corx]{Theorem}

\newtheorem{proposition}[corollary]{Proposition}

\newrepcor{cor}{Corollary}
\newreptheorem{theorem}{Theorem}
\newrepprop{prop}{Proposition}

\newtheorem*{theorem*}{Theorem}
\newtheorem*{proposition*}{Proposition}
\newtheorem{lemma}[corollary]{Lemma}

\theoremstyle{definition}

\newtheorem*{remark*}{Remark}
\newtheorem{definition}[corollary]{Definition}
\newtheorem*{definition*}{Definition}
 
\newtheorem{remark}[corollary]{Remark}

\theoremstyle{remark} 
\numberwithin{equation}{section}
\numberwithin{figure}{section}

\newcommand{\D}{\mathsf{D}}

\newcommand{\diag}{\mathrm{diag}}

\newcommand{\R}{\mathbb{R}}

\newcommand{\dif}{\mathsf{d}}

\newcommand{\pa}{\partial}

\newcommand{\dx}{{\dif x}}

\newcommand{\dist}{{\mathrm{dist}}}
\newcommand{\diam}{{\mathrm{diam}}}
\newcommand{\Aff}{{\mathrm{Aff}}}

\newcommand{\transp}[1]{#1^{\raisebox{\depth}{$\mskip-1mu\scriptstyle\intercal$}}}

\begin{document}

\title{(In)finiteness of slopes for Monge-Amp\`ere equations at flat boundary pieces} 

\author{Xin Nie}
\address{
	Shing-Tung Yau Center and School of Mathematics, Southeast University, Nanjing 210096, China}
\email{nie.hsin@gmail.com}
\thanks{The author is supported by the NSFC grant 12471071.}

\begin{abstract}	
	On any bounded convex domain $\Omega\subset\R^n$ with a $d$-dimensional flat boundary piece $F\subset\pa\Omega$, we consider convex solutions with zero boundary values to the Monge-Amp\`ere equations
	$$
	\det\D^2 u= \dist(\,\cdot\,,\Aff(F))^\alpha,\quad\det\D^2 u=\dist(\,\cdot\,,\pa\Omega)^\alpha,\quad \det\D^2 u=(-u)^\alpha,
	$$
where $\dist$ stands for Euclidean distance and $\Aff(F)$ is the affine hull of $F$. We show that the threshold of the exponent $\alpha\in\R$ for $u$ to have infinite slope at $F$ is $\alpha\leq 2d-n$. 
For the first two equations, we also show that no solution exists when $\alpha\leq 2d-2n$. Moreover, when $\Omega$ satisfies the ``exterior cone condition with a $d$-dimensional ridge'' at $F$, both thresholds are sharp for the first equation, and the infinite-slope threshold is sharp for the last equation.
\end{abstract}
%\setcounter{tocdepth}{1}
%\tableofcontents
\maketitle

\section{Introduction}\label{sec_1}
\subsection{A motivating question}
Throughout this paper, we fix integers $n\geq 2$ and $0\leq d\leq n-1$. By a \emph{$d$-dimensional flat boundary piece} of a bounded convex domain $\Omega\subset\R^n$, we mean the intersection 
$$
F=\pa\Omega\cap L
$$ 
of the boundary $\pa\Omega$ with a $d$-dimensional affine subspace $L\subset\R^n$ such that the relative interior
$$
F^\circ:=\text{ the interior of $F$ as a subset of $L$}
$$
is nonempty. Note that $L$ equals the affine hull $\Aff(F)$ (the smallest affine subspace of $\R^n$ containing $F$) and that a $0$-dimensional flat boundary piece is just any point of $\pa\Omega$. In this paper, we study Alexandrov convex solutions of Monge-Amp\`ere equations on $\Omega$ near flat boundary pieces.

The $d$-dimensional faces of a convex polytope are examples of $d$-dimensional flat boundary pieces. This paper is motivated by the following question raised in \cite[\S 6]{Nie-Seppi_2}:
\begin{quote}
\textbf{Question.} Let $P$ be the interior of a bounded convex polytope $\overline{P}\subset\R^n$ and $u\in C(\overline{P})$ be a nonzero convex function satisfying 
$$
	\det \D^2 u=(-u)^\alpha\  \text{ in $P$},\quad	u|_{\pa P}=0
$$
in the Alexandrov sense, where $\alpha\in\R$. At which points of $\pa P$ does $u$ have infinite slope? 
\end{quote}
Here, a convex function $f\in C(\overline{\Omega})$ on a convex domain $\Omega$ is said to \emph{have finite slope} at a boundary point $x_0\in\pa\Omega$ if it admits a subgradient at $x_0$, namely $f$ is bounded from below by an affine function whose value at $x_0$ coincides with $f(x_0)$. If $f|_F=0$ for a flat boundary piece $F\subset\pa\Omega$, then we have the equivalence
$$
\text{ $f$ has finite slope at some $x_0\in F^\circ$ } \ \Longleftrightarrow\ \text{ $f$ has finite slope at every $x_0\in F$}
$$
(see \S \ref{subsec_MA}). We will refer to this situation simply as ``\emph{$f$ has finite slope at $F$}''. The results of this paper imply:
\begin{quote}
	\textbf{Answer.} $u$ has finite slope exactly at the faces of $P$ of dimension less than $\frac{n+\alpha}{2}$.
\end{quote}

\begin{remark}[Case $\alpha=0$]\label{rk_alpha0}
The case $\alpha=0$ is closely related to a result of Caffarelli \cite{Caffarelli_1993}, stating that a convex subsolution to the equation $\det\D^2u=1$ on a domain in $\R^n$ cannot vanish on an affine subspace of dimension $n/2$ or higher (see also \cite{CaffarelliYuan2022,Mooney_2015,MooneyRakshit2023}). Indeed, the above answer implies this result by the comparison principle, whereas the shorter proof of the result by Mooney \cite{Mooney_2015} can be adapted to give one direction of the answer for $\alpha=0$ (the other, easier direction is provided by a Pogorelov-type barrier; see below). Nevertheless, even in this case, our proof of the former direction is different from Mooney's, thus giving yet another proof of Caffarelli's result. 
\end{remark}

\begin{remark}[Motivation]\label{rk_motivation}
For a solution $u$ to a Monge-Amp\`ere equation on a convex domain, ``having finite slope'' (at a boundary point) is essentially a local Lipschitz property and follows on many occasions from the vast existing work on boundary/global estimates for Monge-Amp\`ere equations 
(see e.g.\ \cite{Le_2, Le_3, Le_4, Le-Savin, Mohammed, Savin-Zhang}), whereas ``having infinite slope'' is a less studied aspect. Our motivation for studying the latter comes from affine differential geometry, where ``having infinite slope everywhere on $\pa\Omega$'' (equivalent to $\lim_{x\to\pa\Omega}\|\D u(x)\|\to+\infty$) corresponds to the \emph{Euclidean completeness} of the Legendre dual hypersurface in $\R^{n+1}$. For example, \emph{hyperbolic affine spheres} and their generalization, convex hypersurfaces of \emph{constant affine Gauss-Kronecker curvature}, are given respectively by convex functions $w,u\in C(\overline{\Omega})\cap C^\infty(\Omega)$ satisfying 
$$
\det \D^2 w=(-w)^{-n-2}\  \text{ in $\Omega$},\quad	w|_{\pa\Omega}=0,
$$
$$
\det \D^2 u=(-w)^{-n-2}\  \text{ in $\Omega$, with $w$ as above};
$$
the infiniteness of slopes has been established for $w$ in \cite{Cheng-Yau}, for $u$ in \cite{Li-Simon-Chen} under smoothness assumptions, and generalized in \cite{Li-Simon-Zhao,  Nie-Seppi_1, Nie-Seppi_3,Nie-Seppi_2, Wu-Zhao, Xiong-Yang}.
\end{remark}

\subsection{Statement of results}
More generally, for any bounded convex domain $\Omega\subset\R^n$ with a flat boundary piece $F\subset\pa\Omega$, consider the three Monge-Amp\`ere equations
\begin{equation}\label{eqn_introequations}
\det\D^2u=\dist(\,\cdot\,,\Aff(F))^\alpha,\quad \det\D^2u=\dist(\,\cdot\,,\pa\Omega)^\alpha,\quad \det\D^2u=(-u)^\alpha,
\end{equation}
where $\dist$ stands for the Euclidean distance and the exponent $\alpha\in\R$ can take both signs. We investigate whether a convex sub- or supersolution has finite slope at $F$. 
%Clearly, the more negative $\alpha$ is, the more likely the slope is to be infinite. Moreover, when $\alpha$ is too negative 
The main results can be summarized as:
\begin{itemize}
	\item as $\alpha$ reaches or falls below $2d-n$, the slope becomes infinite; 
	\item for the first two equations, as $\alpha$ reaches or falls below $2d-2n$, subsolutions no longer exist.
\end{itemize}
The following theorem is a precise statement in one direction. We emphasize that in this paper, every equation or inequality with left-hand side $\det\D^2u$ is understood in the Alexandrov sense.
\begin{thmx}\label{thm_A}
Let $\Omega\subset\R^n$ be a bounded convex domain, $F\subset\pa\Omega$ be a $d$-dimensional flat boundary piece, and $L:=\Aff(F)$ be the affine hull of $F$. 
\begin{enumerate}[label=(\arabic*)]
\item\label{item_thmA1}
Let $\alpha\leq 2d-n$, and let $u\in C(\overline{\Omega})$ be a convex function vanishing on $F$ and satisfying any of the conditions
\begin{itemize}
	\vspace{3pt}
	\item $\det\D^2u\geq M\,\dist(\,\cdot\,,L)^\alpha\ $ in $\Omega$,
	\vspace{2pt}
	\item $\det\D^2u\geq M\,\dist(\,\cdot\,,\pa\Omega)^\alpha\ $ in $\Omega$,
	\vspace{2pt}	
	\item $u<0$ and $\det\D^2u\geq M\,(-u)^\alpha\ $ in $\Omega$,
	\vspace{3pt}	
\end{itemize}
%$$
%\det\D^2u\geq M\,\dist(\,\cdot\,,L)^\alpha,\quad  	\det\D^2u\geq M\,\dist(\,\cdot\,,\pa\Omega)^\alpha,\ \text{ or } \det\D^2u\geq M\,(-u)^\alpha
%$$
where $M>0$. Then $u$ has infinite slope at $F$.
\item\label{item_thmA2}
Let $\alpha\leq 2d-2n$. Then there does not exist any convex function $u\in C(\overline{\Omega})$ satisfying 
$$
\det \D^2u\geq M\,\dist(\,\cdot\,,L)^\alpha\ \text{ in }\Omega
$$
with $M>0$.
\end{enumerate}
\end{thmx} 
The $\alpha=0$ case of part \ref{item_thmA1} implies the result of Caffarelli discussed in Remark \ref{rk_alpha0}.

\begin{remark}[Contrasts between equations]
Although part \ref{item_thmA2} is stated only for the first equation in \eqref{eqn_introequations}, the same conclusion for the second equation follows  because $\dist(\,\cdot\,,\pa\Omega)\leq \dist(\,\cdot\,,L)$ and $\alpha\leq 2d-2n<0$. In contrast, in part \ref{item_thmA1}, $\alpha$ does not have a definite sign, hence there is no direct implication between the statements for different equations. Also note that the third equation does not suffer from existence issues: it is known to admit a convex solution whenever $\alpha\neq n$, while the case $\alpha=n$ amounts to the Monge-Amp\`ere eigenvalue problem  (see e.g.\ \cite{Gu, Le_2} for an overview). Moreover, the solution is known to be unique for $\alpha<n$, while a recent work \cite{Zhou} shows the same for $\alpha>n$.
%Also note that, while part (2) does not require any boundary value assumption beyond continuity, for (1) we need to assume $u|_{F}=0$ so that the rapid growing Monge-Amp\`ere measure near $F$ only contributes to the convexity of $u$ in directions transverse to $F$, rather than along $F$.
\end{remark}

\begin{remark}[Contrast between $\alpha<2d-n$ and $\alpha=2d-n$]\label{rk_critical}
In the easier case $\alpha<2d-n$,  Theorem \ref{thm_A} \ref{item_thmA1} can be proved alternatively via barriers, which give the following more refined estimate for the first two (resp.\ the last) equation in \eqref{eqn_introequations}:
\begin{equation}\label{eqn_introholder}
-u \gtrsim \dist(\,\cdot\,,L)^\frac{\alpha+2n-2d}{n}\quad\Big(\text{resp. }-u \gtrsim \dist(\,\cdot\,,L)^\frac{2n-2d}{n-\alpha}\Big)\ \ \text{ in }\Gamma \text{ near }F,
\end{equation}
where $\Gamma\subset\R^n$ is any cone with apex on $F$ and contained in $\Omega$ near the apex. When $\alpha=2d-n$, so that the exponents in \eqref{eqn_introholder} are $1$, these estimates still hold but fail to be sharp and do not imply the infiniteness of slopes. We expect a sharp estimate with a logarithmic factor to do so, as already established for $d=n-1$ by Le \cite{Le_4} and Gu \cite{Gu} (see \S \ref{subsec_related} below).
%\ref{sec_open}
\end{remark}
\begin{figure}[ht]
	\centering
	\includegraphics[width=0.89\textwidth]{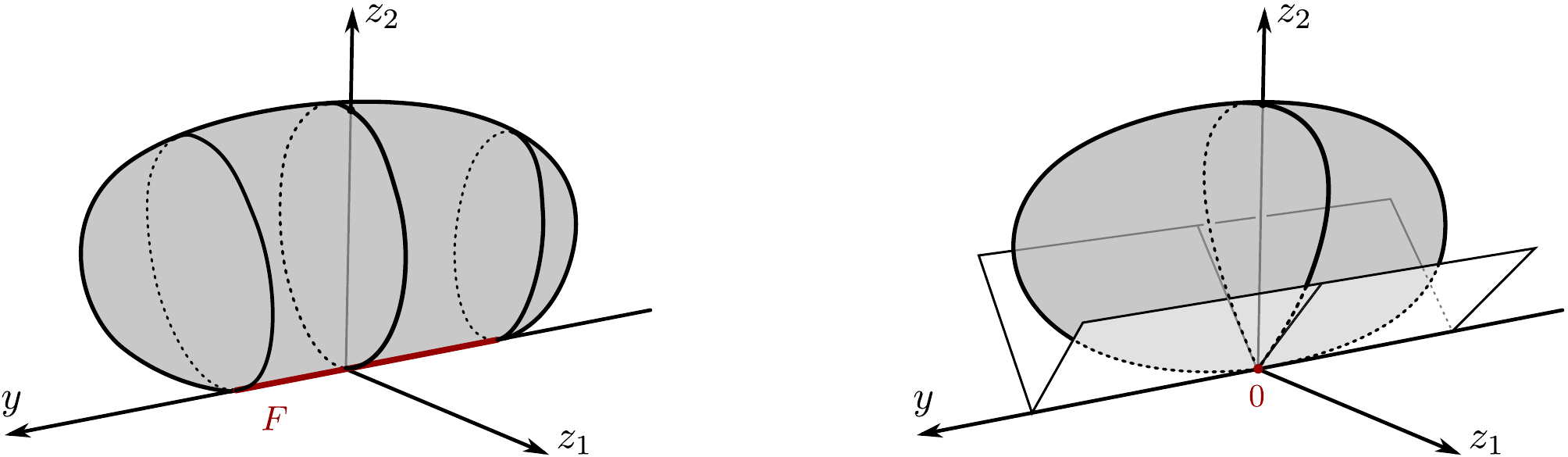}
	\caption{Convex domains considered in Theorem \ref{thm_A} (left) and Theorem \ref{thm_B} (right), for $n=3$ and $d=1$. The figure also shows the coordinate system $x=(y,z)$ that we use in the proofs, where $y\in\R^d$, $z\in\R^{n-d}$, $L=\{z=0\}$, and $x_0=(0,0)$.}
	\label{fig_domains}
\end{figure}

In the reverse direction, we want to show that the thresholds $2d-n$ and $2d-2n$ are sharp, but it is clearly not the case under the same assumptions. For example, any flat boundary piece $F$ of dimension $d>0$ contains infinitely many flat boundary pieces $F'$ of dimension $d'<d$, but applying Theorem \ref{thm_A} to $F'$ leads to non-sharp conclusions. 
For the reverse controls to hold, we shall consider, rather than a $d$-dimensional flat boundary piece, a boundary point $x_0\in\pa\Omega$ with the property that $\Omega$ is contained in a convex cone whose \emph{ridge} has dimension $d$ and contains $x_0$ (see \S \ref{subsec_coordinates}). Another subtlety is that our proof of the reverse control (explicit barriers) does not work for the second equation in \eqref{eqn_introequations}. Thus, we have the following partial converse to Theorem \ref{thm_A}:
\begin{thmx}\label{thm_B}
Let $\Omega\subset\R^n$ be a bounded convex domain contained in a convex cone $\Gamma\subset\R^n$ such that the ridge $L$ of $\Gamma$ has dimension $d$ and contains some $x_0\in\pa\Omega$.
\begin{enumerate}[label=(\arabic*)]
	\item\label{item_thmB1}
Let  $\alpha> 2d-n$, and let $u\in C(\overline{\Omega})$ be a convex function vanishing on $\pa\Omega$ and satisfying either of the conditions
\begin{itemize}
	\vspace{3pt}
	\item $\det\D^2u\leq M\,\dist(\,\cdot\,,L)^\alpha\ $ in $\Omega$, 
	\vspace{2pt}	
	\item $u<0$ and $\det\D^2u\leq M\,(-u)^\alpha\ $ in $\Omega$,
	\vspace{3pt}	
\end{itemize}
where $M>0$. Then $u$ has finite slope at $x_0$. 
%Moreover, if we further assume $\alpha>-2$, the same conclusion holds under the condition 
%$$
%\det\D^2u\leq M\,\dist(\,\cdot\,,\pa\Omega)^\alpha\ \text{ in } \Omega.
%$$
	\item\label{item_thmB2}
	 Let $\alpha>2d-2n$, and let $g:\Omega\to\R_{\geq0}$ be a nonnegative measurable function satisfying
$$
	 g(x)\leq M\,\dist(x,L)^\alpha\ \ \text{ for all }x\in\Omega
$$
	 with $M>0$. Then there exists a unique convex solution $u\in C(\overline{\Omega})$ to the Dirichlet problem 
	\begin{equation}\label{eqn_thmB2}
	\det\D^2u=g\  \text{ in }\Omega,\quad u|_{\pa\Omega}=0.
	\end{equation}
\end{enumerate}
\end{thmx}
\begin{remark}[Other existence results]
The classical existence theorem states that a sufficient condition for \eqref{eqn_thmB2} to admit a convex solution is $\int_\Omega g(x)\dif x<+\infty$. A new existence result of Le \cite[Theorem 1.1]{Le_5} implies that the condition can be weakened to 
\begin{equation}\label{eqn_Leexistence}
\int_\Omega g(x)\,\dist(x,\pa\Omega)\dif x<+\infty.
\end{equation}
Theorem \ref{thm_B} \ref{item_thmB2} is not covered by this result. Indeed, if $g(x)=\dist(x,L)^\alpha$ and $\Omega$ is the interior of $\Gamma\cap B_r(x_0)$, where $B_r(x_0)$ is a ball centered at $x_0$, then \eqref{eqn_Leexistence} holds if and only if $\alpha>d-n-1$, which is stronger than our threshold $\alpha>2d-2n$ unless $d=n-1$.
\end{remark}
Figure \ref{fig_domains} illustrates convex domains satisfying the assumptions of
Theorems \ref{thm_A} and \ref{thm_B}, respectively. When they are satisfied at the same time, the theorems imply that both thresholds are sharp for the first and last equations in \eqref{eqn_introequations}: 
\begin{corx}\label{cor_C}
Let $\Omega\subset\R^n$ be a bounded convex domain contained in a convex cone $\Gamma\subset\R^n$ such that the ridge $L$ of $\Gamma$ has dimension $d$ and contains a flat boundary piece $F\subset\pa\Omega$ as a subset with nonempty interior. 
\begin{enumerate}[label=(\arabic*)]
	\item There exists a convex function $u\in C(\overline{\Omega})$ satisfying
	$$
	\det \D^2u\asymp \dist(\,\cdot\,,L)^\alpha\ \text{ in }\Omega, \quad u|_{\pa\Omega}=0
	$$
	if and only if $\alpha>2d-2n$. Such a function $u$ has finite slope at $F$ if and only if $\alpha>2d-n$.
	\item\label{item_corC2}
	Let $u\in C(\overline{\Omega})$ be a nonzero convex function satisfying
$$
\det \D^2u\asymp (-u)^\alpha\ \text{ in }\Omega, \quad u|_{\pa\Omega}=0
$$
for some $\alpha\in\R$. Then $u$ has finite slope at $F$ if and only if $\alpha>2d-n$.
\end{enumerate}
\end{corx}
Part \ref{item_corC2} implies the answer to the question at the beginning of the paper. It would be interesting to investigate whether the same conclusion holds for the equation $\det\D^2u=\dist(\,\cdot\,,\pa\Omega)^\alpha$.
%(see \S \ref{sec_open}).

\subsection{Other related work}\label{subsec_related}
Below, we review some related work in addition to that already mentioned.
%Remarks \ref{rk_alpha0} and \ref{rk_motivation}.
\begin{itemize}[leftmargin=0.6cm]
\item The $n=2$ case of Theorem \ref{thm_A} \ref{item_thmA1} and Theorem \ref{thm_B} \ref{item_thmB1} are contained in \cite[\S 8.2]{Nie-Seppi_1}.
\item 
The works of Le \cite{Le_4} and Gu \cite{Gu} provide quantitative versions of Theorem \ref{thm_A} \ref{item_thmA1} and Theorem \ref{thm_B} \ref{item_thmB1} in the case $d=n-1$. More specifically:
\begin{itemize}[leftmargin=0.4cm]
	\item The cone condition in Theorem \ref{thm_B} is vacuous in this case. 
Meanwhile, since $\dist(\,\cdot\,,\pa\Omega)^\alpha\leq\dist(\,\cdot\,,L)^\alpha$ for $\alpha>n-2\geq 0$, the statement of Theorem \ref{thm_B} \ref{item_thmB1} for the right-hand side $\dist(\,\cdot\,,L)^\alpha$ implies the one for $\dist(\,\cdot\,,\pa\Omega)^\alpha$ in this case.
Therefore, the statement implies that a supersolution with zero boundary values to either of the last two equations in \eqref{eqn_introequations} with $\alpha>n-2$
has finite slope everywhere on $\pa\Omega$. This is a qualitative consequence of the Lipschitz estimates in \cite[Theorem 1.1]{Le_4}.
	\item Theorem \ref{thm_A} \ref{item_thmA1} implies that a subsolution with zero boundary values to either of the last two equations in \eqref{eqn_introequations} with $\alpha\leq n-2$ has infinite slope at any $(n-1)$-dimensional flat boundary piece. As mentioned in Remark \ref{rk_critical}, the more difficult $\alpha=n-2$ part of this statement is a qualitative consequence of the log-Lipschitz estimates in \cite{Gu, Le_4}. For the equation with right-hand side $(-u)^{n-2}$, their estimate is
$$
-u\gtrsim \dist(\,\cdot\,,L)\,|\log \dist(\,\cdot\,,L)|^\frac{n}{2}\ \ \text{ in }\Omega \text{ near } F.
$$
The exponent $\frac{n}{2}$ is replaced by $1$ for the equation with right-hand side $\dist(\,\cdot\,,\pa\Omega)^{n-2}$, and ``$\gtrsim$'' is replaced by ``$\lesssim$'' if we consider supersolutions instead of subsolutions.
\end{itemize}
%When $d=n-1$, the work of Le \cite{Le_4} and Gu \cite{Gu} can be viewed as a quantitative version of Theorem \ref{thm_A} \ref{item_thmA1} and Theorem \ref{thm_B} \ref{item_thmB1}. Note that the cone condition in Theorem \ref{thm_B} is vacuous in this case.
% established a similar estimate with right-hand side $\dist(\,\cdot\,,L)\,|\log \dist(\,\cdot\,,L)|^\frac{n}{2}$.

\item Jian and Wang \cite{Jian-Wang} considered a finer condition on $\Omega$ than the cone condition in Theorem \ref{thm_B}, which consists in replacing the proper factor $\Gamma_0$ of our bounding cone  $\Gamma\cong\R^d\times\Gamma_0$ ($\Gamma_0$ is a proper convex cone in $\R^{n-d}$; see \S \ref{subsec_coordinates}) by a convex set of the form
$$
\Bigg\{z=(z_1,\cdots,z_{n-d})\in\R^{n-d}\,\Bigg|\,z_{n-d}\geq\sum_{i=1}^{n-d-1}C_i|z_i|^{a_i}\Bigg\}
$$
with $a_i\geq1$. The case $a_1=\cdots=a_{n-d-1}=1$ boils down to our condition. Their main result, when specialized to this case and to the equations $\det\D^2u=\dist(\,\cdot\,,\pa\Omega)^\alpha$ and $\det\D^2u=(-u)^\alpha$ with $\alpha<2d-n$, is exactly the reverse estimate to \eqref{eqn_introholder}.  

%can be stated in the following form, which may be compared with ours:
%\begin{itemize}
% \item if $\alpha<2d-n$, then $u\in C^\mu(\overline{\Omega})$ for $\mu:=\frac{2n-2d}{n-\alpha}\in(0,1)$;
% \item if $\alpha\geq 2d-n$, then $u\in C^\lambda(\overline{\Omega})$ for every $\lambda\in(0,1)$.
%\end{itemize}
%Note the difference from our dichotomy: $\alpha\leq 2d-n$ versus $\alpha>2d-n$ (cf.\ Remark \ref{rk_critical}). 

\item
The result of Caffarelli mentioned in Remark \ref{rk_alpha0} is generalized by Jin, Tu and Xiong \cite{JinTuXiong2025SingularSet} to the Monge-Amp\`ere obstacle problem $\det\D^2u= u^\alpha\chi_{\{u>0\}}$ with $\alpha\in[0,n)$. Their optimal dimension bound $d<\frac{n+\alpha}{2}$ for singular flat pieces is the same as our finite-slope threshold $\alpha>2d-n$. However, since their $u$ is nonnegative, this problem has an essential difference with our equation $\det\D^2u=(-u)^\alpha$ with nonpositive $u$, except in the case $\alpha=0$ discussed in Remark \ref{rk_alpha0}. The common threshold arises from the same scaling law and dimensionality.
\end{itemize}
\subsection{Sketch of proofs and organization of the paper}
The proofs of the four parts of the two theorems are independent of each other. Theorem \ref{thm_A} \ref{item_thmA1} is proved by contradiction, based on the convergence of dilated $u$'s toward directional derivatives. The familiar version of this convergence says that if $u\in C(\overline{\Omega})$ is convex, vanishes at $x_0\in\pa\Omega$, and has finite slope at $x_0$, then as $\lambda\to+\infty$,
$$
u_\lambda(x):=\lambda\, u\!\left(x_0+\frac{x-x_0}{\lambda}\right)
$$
converges locally uniformly to a $1$-homogeneous convex function $u_\infty$, namely the directional derivatives, defined on the tangent cone of $\Omega$ at $x_0$. We study a variant of this construction, where $x_0$ is replaced by a flat boundary piece $F$ of any dimension and we only scale in the directions transverse to $F$. 
The contradiction argument proceeds by showing that, on the one hand, $u_\infty$ has vanishing Monge-Amp\`ere measure; on the other hand, when $\alpha\leq 2d-n$, the hypothetical inequalities on $\det \D^2u$ prevent $\det \D^2u_\lambda$ from tending to $0$.

For Theorem \ref{thm_A} \ref{item_thmA2}, we use the following estimate for the Monge-Amp\`ere measure $\mathcal{M}u$ of a convex function $u$ defined on a neighborhood of the origin in a convex cone $\Gamma\subset\R^n$: when a compact set $K\subset\Gamma^\circ$ is shrunk by a factor of $t\to0^+$, its mass under $\mathcal{M}u$ satisfies
$$
\mathcal{M}u(tK)=o(t^{-n}).
$$
If $\Gamma$ has a $d$-dimensional ridge $L$ and $K$ is shrunk only in the directions transverse to $L$, the right-hand side improves to $o(t^{d-n})$ and implies the required conclusion.

The two parts of Theorem \ref{thm_B} are both proved with explicit barriers. Part \ref{item_thmB1} uses the following generalization of the classical Pogorelov-type barrier of Caffarelli \cite{Caffarelli_1993}:
\begin{equation}\label{eqn_introbarrier1}
	f(y,z)=(|y|^2+\Lambda)|z|^\frac{\alpha+2n-2d}{n}\ \  \text{ for } (y,z)\in\R^d\times\R^{n-d}=\R^n, 
\end{equation}
where $\Lambda>0$. When $\alpha>2d-n$, $f$ is convex and satisfies $\det\D^2f\geq |z|^\alpha$ in the slab $\{|y|<R\}$ for some $R=R(\Lambda)>0$.  
To apply the barrier to the equation $\det\D^2u=(-u)^\alpha$, we need the subcritical comparison principle of Le \cite{Le_4} (see Lemma \ref{lemma_Lecomparison}). For part \ref{item_thmB2}, we construct a new barrier of a similar form 
\begin{equation}\label{eqn_introbarrier2}
	f(y,z)=(|y|^2-\Lambda)(z_1\cdots z_{n-d})^\frac{\alpha+2n-2d}{n(n-d)},
\end{equation}
which now requires $2d-2n<\alpha< 2d-n$ and is convex only in $\{|y|<R,\,z_1,\cdots, z_{n-d}\geq0\}$.

In the rest of the paper, we first collect some preliminaries in \S \ref{sec_preliminaries}, then we present the above proofs of Theorem \ref{thm_A} \ref{item_thmA1}, Theorem \ref{thm_A} \ref{item_thmA2} and Theorem \ref{thm_B} in \S \ref{sec_infinite},  \S \ref{sec_nonexistence} and  \S \ref{sec_finiteslope}, respectively.

\subsection*{Acknowledgments} We are grateful to Mengni Li for enlightening discussions which brought our attention back to the problem in \cite[\S 6]{Nie-Seppi_2}.

\subsection*{Statement on the use of AI}
The finite- and infinite-slope results were obtained by the author in late 2025, with the finite-slope part based on more complicated barriers on simplices. When revisiting the work with the assistance of ChatGPT 5.6 Sol in August 2026, we obtained the nonexistence result for $\alpha\leq 2d-2n$ and the simpler barriers \eqref{eqn_introbarrier1} and \eqref{eqn_introbarrier2}.
The author independently verified every mathematical suggestion and derivation the model produced. 
When preparing this manuscript, the exposition and all mathematical proofs were written entirely by the author, and AI was used only for literature searches and final checks. 
The author takes full responsibility for the content of this paper.

\section{Preliminaries}\label{sec_preliminaries}
We begin in \S \ref{subsec_MA} by reviewing some background on Monge-Amp\`ere equations and finiteness of slopes. In \S \ref{subsec_comparison}, we present the comparison principle of Le \cite{Le_4} for the equation $\det\D^2u=(-u)^\alpha$ with $\alpha<n$. Finally, in \S \ref{subsec_coordinates}, we fix some terminology and notation concerning convex cones.

\subsection{Convex functions and Monge-Amp\`ere measures}\label{subsec_MA}
By a \emph{convex domain}, we mean a nonempty convex open set.
Given a convex function $u:\Omega\to\R$ on a convex domain $\Omega\subset\R^n$, a \emph{subgradient} of $u$ at $x\in \Omega$ is by definition an element of the set
\begin{equation}\label{eqn_definesubgradient}
\pa u(x):=\big\{a\in\R^n\,\big|\,u(x')\geq u(x)+a\cdot(x'-x)\ \text{ for all }x'\in \Omega\}.
\end{equation}
The \emph{Monge-Amp\`ere measure} $\mathcal{M}u$ of $u$ is the Borel measure on $\Omega$ defined by
$$
\mathcal{M}u(E):=|\pa u(E)|
$$
for any Borel set $E\subset\Omega$, where we denote
$$
\pa u(E):=\bigcup_{x\in E}\pa u(x),
$$
and let $|\cdot|$ denote the volume with respect to the Lebesgue measure $\dx$ on $\R^n$. If $u\in C^2(\Omega)$, we have
$$
\mathcal{M}u(E)=\int_E\det\D^2u(x)\dx,
$$
namely, the Hessian determinant $\det\D^2u$ represents the density of $\mathcal{M}u$ with respect to the Lebesgue measure. This allows us to interpret equations or inequalities with left-hand side $\det\D^2u$ \emph{in the Alexandrov sense}, i.e., as comparison relations between measures. In particular, a \emph{subsolution} (resp.\ \emph{supersolution}) to a Monge-Amp\`ere equation of the form $\det\D^2u=F(x,u)$ is by definition a convex function $u$ satisfying
$\det\D^2u\geq F(x,u)$ (resp.\ $\det\D^2u\leq F(x,u)$) in the Alexandrov sense. We often use both terms up to a multiplicative constant, e.g.\ loosely calling $u$ a subsolution when $\det\D^2u\geq MF(x,u)$ with a constant $M>0$. 

We collect below some standard results about Monge-Amp\`ere measures for later use. Here, we fix a bounded convex domain $\Omega\subset\R^n$. Most of the results are formulated for convex functions in $C(\overline{\Omega})$, i.e., those which extend continuously to the boundary $\pa\Omega$.

\begin{lemma}[Alexandrov maximum principle; see {\cite[Theorem 1.4.2]{Gutierrez}}]\label{lemma_AlexandrovMP}
For any convex function $u\in C(\overline{\Omega})$ such that $u|_{\pa\Omega}=0$, we have
$$
		|u(x)|^n\leq C_n\,\diam(\Omega)^{n-1} \dist(x,\pa\Omega)\mathcal{M}u(\Omega),
$$
where the constant $C_n>0$ only depends on the dimension $n$, and $\diam(\Omega)$ denotes the diameter of $\Omega$.
\end{lemma}

\begin{lemma}[Comparison principle; see {\cite[Theorem 1.4.6]{Gutierrez}}]\label{lemma_comparison}
If two convex functions $u,v\in C(\overline{\Omega})$ satisfy
$$
u|_{\pa\Omega}\geq v|_{\pa\Omega},\qquad \det\D^2u\leq \det\D^2v \ \text{ in }\Omega,
$$
then we have $u\geq v$ in $\Omega$.
\end{lemma}

\begin{lemma}[Existence of solutions; see {\cite[Theorem 1.6.2]{Gutierrez}} and  \cite{Hartenstine2006}]\label{lemma_existence}
Let $\mu$ be a Borel measure on $\Omega$ with total mass $\mu(\Omega)<+\infty$ and $\varphi:\pa\Omega\to\R$ be a function which can be extended to a continuous convex function on $\overline{\Omega}$ (cf.\ Remark \ref{rk_extension} below). Then there exists a convex function $u\in C(\overline{\Omega})$ such that
$$
\det\D^2u=\mu\ \text{ in }\Omega,\quad u|_{\pa\Omega}=\varphi.
$$
\end{lemma}

\begin{lemma}[Stability of Monge-Amp\`ere measure; see {\cite[Lemma 1.2.3]{Gutierrez}}]\label{lemma_stability}
Let $u_k:\Omega\to\R$, $k=1,2,\cdots$ be a sequence of convex functions converging to a convex function $u_\infty:\Omega\to\R$ uniformly on compact subsets of $\Omega$. Then the Monge-Amp\`ere measure $\mathcal{M}u_k$ converges weakly to $\mathcal{M}u_\infty$.
\end{lemma}

The last result is often used in combination with the following fact about convex functions:
\begin{lemma}[Uniform convergence of convex functions; see {\cite[Theorem 10.8]{Rockafellar}}]\label{lemma_uniform}
Let $u_k:\Omega\to\R$, $k=1,2,\cdots$ be a sequence of convex functions pointwise converging to a function $u_\infty:\Omega\to\R$. Then $u_\infty$ is convex and the convergence $u_k\to u_\infty$ is uniform on compact subsets of $\Omega$.
\end{lemma}

While the Monge-Amp\`ere measure is defined only in the interior of the domain of a convex function, the notion of subgradients makes sense also on the boundary. Namely, the definition \eqref{eqn_definesubgradient} remains valid when $x\in\pa\Omega$ for a convex function $u\in C(\overline{\Omega})$, even though it is now possible that $\pa u(x)=\varnothing$. We will denote a boundary point by $x_0\in\pa\Omega$ instead.
The existence of subgradients at $x_0$ can be characterized in terms of the one-sided directional derivative
\begin{equation}\label{eqn_slope}
	\lim_{t\to 0^+}\frac{u\big(x_0+t(x_1-x_0)\big)-u(x_0)}{t}\in[-\infty,+\infty),
\end{equation}
where $x_1\in\Omega$ is a point in the interior, and the limit exists (possibly being $-\infty$) because by convexity, the fraction is decreasing as $t$ goes to $0^+$. More precisely, we have:
\begin{proposition}[{\cite[Proposition 4.13]{Nie-Seppi_1}}]\label{prop_slopes}
Given a convex function $u\in C(\overline{\Omega})$ and a boundary point $x_0\in\pa\Omega$, the following conditions are equivalent to each other:
\begin{itemize}
	\item The limit \eqref{eqn_slope} is finite for some $x_1\in\Omega$.

	\item The limit \eqref{eqn_slope} is finite for all $x_1\in\Omega$.
	\item $u$ admits a subgradient at $x_0$, i.e., $\pa u(x_0)\neq\varnothing$.
	\item There exists a sequence $(x_k)_{k=1,2,\cdots}$ of points in $\Omega$ tending to $x_0$ such that $u$ is differentiable\footnote{A convex function $u:\Omega\to\R$ is differentiable at $x\in\Omega$ if and only if $\pa u(x)$ consists of a single point, namely the gradient $\D u(x)$. The set of points where $u$ is not differentiable has measure zero. See \cite[\S 25]{Rockafellar}.} at $x_k$ and $|\D u(x_k)|$ does not tend to $+\infty$.
\end{itemize}
\end{proposition}
Thus, we may introduce the main terminology of this paper formally as:
\begin{definition}[(In)finiteness of slopes]\label{def_slopes}
In the setting of Proposition \ref{prop_slopes}, $u$ is said to \emph{have finite slope at $x_0$} if the conditions hold; otherwise, it is said to \emph{have infinite slope at $x_0$}.
\end{definition}
As pointed out in the introduction, if $u$ vanishes on a flat boundary piece $F\subset\pa\Omega$, then having finite slope at a point $x_0$ in the relative interior $F^\circ$ of $F$ implies the finiteness of slopes everywhere on $F$. Indeed, $a\in\R^n$ is a subgradient of $u$ at $x_0$ if and only if $u$ is bounded from below by an affine function $h:\R^n\to\R$ of the form $h(x)= a\cdot x+b$ ($b\in\R$) with $h(x_0)=u(x_0)$. The restriction of $h$ to the affine subspace $L:=\Aff(F)$ is an affine function on $L$ which is nonpositive on the open subset $F^\circ\subset L$ and vanishes at $x_0\in F^\circ$; such a function must be zero. Thus, $h$ coincides with $u$ at any point of $F$, and hence $a$ is a subgradient of $u$ at that point.

\begin{remark}[$u\in C(\overline{\Omega})$ is inessential]\label{rk_extension}
As in much of the literature on Monge-Amp\`ere equations, we assume $u\in C(\overline{\Omega})$ in the above discussions for simplicity, but this assumption is not essential. In fact, every convex function $u:\Omega\to\R$ can be uniquely  extended to a lower semicontinuous convex function $u:\overline{\Omega}\to(-\infty,+\infty]$ by setting its boundary values to be
$$
u|_{\pa\Omega}:\pa\Omega\to (-\infty,+\infty],\qquad u(x_0):=\liminf_{x\in\Omega,\,x\to x_0}u(x).
$$
This extension belongs to $C(\overline{\Omega})$ if and only if $u|_{\pa\Omega}\in C(\pa\Omega)$ (which requires $u|_{\pa\Omega}<+\infty$ in particular); see e.g.~\cite[\S 4]{Nie-Seppi_1}. Most of the discussions above actually only require $u|_{\pa\Omega}<+\infty$, and do not need $u|_{\pa\Omega}$ to be continuous. 
For example, in Lemma \ref{lemma_existence}, we may only assume that $\varphi:\pa\Omega\to\R$ is the boundary value of some convex function $\Omega\to\R$ in the above sense (although the resulting solution $u$ might not be in $C(\overline{\Omega})$).
This assumption is equivalent to the condition that $\varphi$ is lower semicontinuous and restricts to a convex function on every line segment contained in $\pa\Omega$ (see \cite[Lemma 4.6]{Nie-Seppi_1}).
\end{remark}

\subsection{Subcritical comparison principle}\label{subsec_comparison}
For sub- and supersolutions to the equation $\det\D^2u=(-u)^\alpha$, the usual comparison principle (Lemma \ref{lemma_comparison}) does not apply, but we have the following alternative:
\begin{lemma}[Subcritical comparison principle]\label{lemma_Lecomparison}
Let $\Omega\subset\R^n$ be a bounded convex domain and $u,v\in C(\overline{\Omega})$ be convex functions such that 
$$
0\geq u\geq v \ \text{ on } \pa\Omega,\quad\det\D^2u\leq M(-u)^\alpha\ \text{ and }\ \det\D^2v\geq M(-v)^\alpha \ \text{ in }\Omega
$$
for $\alpha<n$ and $M>0$. Moreover, assume $v<0$ in $\Omega$. Then we have $u\geq v$ in $\Omega$.
\end{lemma}
This result is classical when $\alpha\leq 0$ (see the proof below) and is due to Le \cite[Lemma 2.1]{Le_4} when $0<\alpha<n$ (see also \cite[Theorem 1.5]{Le_5} for a generalization). We present below a self-contained proof for completeness. Note that the conclusion fails when $\alpha\geq n$, because scaling up $u$ or scaling down $v$ preserves the inequalities on Monge-Amp\`ere measures in this case. For example, if $\det\D^2v\geq M(-v)^\alpha$, $\alpha\geq n$ and $0<\lambda\leq1$, then
\begin{equation}\label{eqn_scaledownv}
\det \D^2(\lambda v)=\lambda^n\det\D^2v\geq M\lambda^n(-v)^\alpha=M\lambda^{n-\alpha}(-\lambda v)^\alpha\geq M(-\lambda v)^\alpha.
\end{equation}
\begin{proof}	
	
\textbf{Case $\alpha<0$.} The usual comparison principle (Lemma \ref{lemma_comparison}) has a classical generalization to sub- and supersolutions of the equation $\det\D^2u=F(x,u)$ with right-hand side $F(x,u)$  non-decreasing in $u$ (see e.g.\ \cite[Prop.\ 2]{Cheng-Yau}, \cite[Lemma 6.4]{Nie-Seppi_2}), from which the required statement for $\alpha\leq 0$ follows. For completeness, we now give a simplified proof for the case $\alpha<0$, where $(-u)^\alpha$ is strictly increasing in $u$.

Suppose by contradiction that $u<v$ somewhere in $\Omega$. We may pick $\epsilon>0$ such that
$$
E:=\{x\in\overline{\Omega}\mid u(x)< v(x)-\epsilon\}
$$
is nonempty. Since $u|_{\pa\Omega}\geq v|_{\pa\Omega}$, we have $\overline{E}\subset\Omega$ and $u=v-\epsilon$ on $\pa E$. Then, on the one hand, since $u=v-\epsilon$ on $\pa E$ and $u\leq v-\epsilon$ in $E$, the mass of the Monge-Amp\`ere measure of $u$ on $E$ is no less than that of $v-\epsilon$ (see e.g.\ \cite[Lemma 1.4.1]{Gutierrez}):
$$
\mathcal{M}u(E)\geq \mathcal{M}(v-\epsilon)(E)=\mathcal{M}v(E);
$$ 
on the other hand, by the hypothesis on $\det\D^2u$ and $\det \D^2v$, we have
$$
\mathcal{M}u(E)\leq M\int_{E}(-u(x))^\alpha\dif x< M\int_E(-v(x))^\alpha\dif x\leq \mathcal{M}u(E)
$$
(the ``$<$'' uses the conditions $\alpha<0$, $v<0$ in $\Omega$, and $\overline{E}\subset\Omega$, the last two ensuring that the integrals are finite), a contradiction.

\textbf{Case $\alpha\geq0$.} Let us first give a proof under the extra assumption $v|_{\pa\Omega}<0$. Suppose by contradiction that $u<v$ somewhere in $\Omega$. By assumptions (note that $u$ and $v$ are nonpositive), the ratio $u/v$ is continuous on $\overline{\Omega}$, and we have $u/v\leq 1$ on $\pa\Omega$ and $u/v>1$ somewhere in $\Omega$. Therefore, $u/v$ attains its maximum
$$
\Lambda:=\max_{x\in\overline{\Omega}}\frac{u(x)}{v(x)}>1
$$
at some $x_0\in\Omega$. The function $w:=\Lambda v$ has the following properties:
\begin{enumerate}[label=(\alph*)]
\item\label{item_leproof1}
$u\geq w$ in $\Omega$, with equality attained at $x_0$;
\item\label{item_leproof2} 
$u>w$ on $\pa\Omega$;
\item\label{item_leproof3} 
$\det \D^2w\geq M(-w)^\alpha$.
\end{enumerate}
They follow respectively from the definition of $w$, the inequalities $u\geq v$ and $v>w$ on $\pa\Omega$ (the latter uses $v|_{\pa\Omega}<0$), and a calculation similar to \eqref{eqn_scaledownv}.
By property \ref{item_leproof2}, we may pick $\epsilon>0$ such that 
\begin{equation}\label{eqn_leproof1}
u\geq w+\epsilon\ \text{ on }\pa\Omega.
\end{equation}
Since $(-u)^\alpha$ is non-increasing in $u$, it follows from properties \ref{item_leproof1} and \ref{item_leproof3} that
\begin{equation}\label{eqn_leproof2}
\det \D^2u\leq M(-u)^\alpha\leq M(-w)^\alpha\leq \det \D^2w=\det \D^2(w+\epsilon)
\end{equation}
By the usual comparison principle (Lemma \ref{lemma_comparison}), \eqref{eqn_leproof1} and \eqref{eqn_leproof2} imply that $u\geq w+\epsilon$ in $\Omega$, contradicting property \ref{item_leproof1}. This completes the proof when $v|_{\pa\Omega}<0$. 

The general case can be reduced to the $v|_{\pa\Omega}<0$ case by dilating $v$ as follows. Again, suppose by contradiction that $u<v$ somewhere in $\Omega$. By translation, we may assume $0\in\Omega$ and $u(0)<v(0)$. Dilating the domain of $v$ radially by a factor of $\lambda>1$ and restricting the dilated function to $\Omega$, we obtain the function
$$
v_\lambda:\overline{\Omega}\to \R,\quad v_\lambda(x):=v(\lambda^{-1}x).
$$
Since $v(0)<0$, we have, by convexity,
\begin{equation}\label{eqn_leproof2.1}
v_\lambda(x)=v\big((1-\lambda^{-1})0+\lambda^{-1}x\big)\leq (1-\lambda^{-1})v(0)+\lambda^{-1}v(x)<\lambda^{-1}v(x)\ \text{ for all } x\in\overline{\Omega}.
\end{equation}
In particular, $v_\lambda|_{\pa\Omega}<0$.
Meanwhile, the condition $\det\D^2v\geq M(-v)^\alpha$ implies 
\begin{equation}\label{eqn_leproof3}
\det \D^2v_\lambda(x)=\lambda^{-2n}(\det \D^2v)(\lambda^{-1}x)\geq \lambda^{-2n}M(-v_\lambda(x))^\alpha\ \text{ for all } x\in\Omega.
\end{equation}
Replacing $v$ by $v_\lambda$, we may repeat the above derivation of contradiction for $v|_{\pa\Omega}<0$ as long as we pick $\lambda>1$ small enough such that properties \ref{item_leproof1}, \ref{item_leproof2} and \ref{item_leproof3} hold. More specifically, by first letting $\lambda$ be smaller than
$$
\frac{u(0)}{v_\lambda(0)}=\frac{u(0)}{v(0)}>1
$$
and noting that $u/v_\lambda<\lambda$ on $\pa\Omega$ by \eqref{eqn_leproof2.1}, we may assume that $u/v_\lambda$ attains its maximum $\Lambda>1$ at some $x_0\in\Omega$ and set $w:=\Lambda v_\lambda$ as before. We then adjust $\lambda$ and check those properties (with $v$ replaced by $v_\lambda$):
\begin{itemize}
\item By \eqref{eqn_leproof2.1} and the definition of $\Lambda$ and $w$, properties \ref{item_leproof1} and \ref{item_leproof2} hold.
%\item 
%Since $u(0)<v(0)<0$, we may first pick $\lambda>1$ small enough such that 
%$$
%	\Lambda\geq\frac{u(0)}{v_\lambda(0)}=\frac{u(0)}{\lambda v(0)}>1.
%$$
\item 
By \eqref{eqn_leproof3}, we have
$$
\det \D^2w= \Lambda^n\det\D^2 v_\lambda\geq \Lambda^n\lambda^{-2n}M(-v_\lambda)^\alpha=M\lambda^{-2n}\Lambda^{n-\alpha}(-w)^\alpha.
$$
Since $\alpha<n$ and $\Lambda>1$, by decreasing $\lambda$ (which increases $\Lambda$) if necessary, we may arrange that $\lambda^{-2n}\Lambda^{n-\alpha}\geq1$, so that property \ref{item_leproof3} holds.
\end{itemize}
Now that the three properties are established, the rest of the argument is the same as that for $v|_{\pa\Omega}<0$. 
\end{proof}

\subsection{Convex cones}\label{subsec_coordinates}
The following definitions and facts about convex cones are standard (see e.g.\ \cite[\S 8]{Rockafellar}), apart from minor differences in terminology\footnote{What we call a ``convex cone'' here is more specifically a ``full dimensional closed convex cone'', and what we call ``ridge'' is usually referred to as ``lineality space'' in convex analysis.}: 
\begin{itemize}[leftmargin=1cm]
\item 
By a \emph{convex cone} in the linear space $\R^n$, we mean a closed convex subset with nonempty interior and invariant under positive scaling. 
\item 
Given a convex cone $\Gamma\subset\R^n$, we refer to its intersection with the opposite cone, namely
$$
L_\Gamma:=\Gamma\cap (-\Gamma),
$$
as the \emph{ridge} of $\Gamma$. It is the largest linear subspace of $\R^n$ contained in $\Gamma$.
\item 
$\Gamma$ is said to be \emph{proper} if $L_\Gamma=\{0\}$. This is equivalent to 
the condition that $\Gamma$ does not contain any entire straight line, and also equivalent to the existence of an affine hyperplane $H\subset\R^n$ not passing through $0$ such that $H\cap\Gamma$ is nonempty and compact. In this case, $H\cap\Gamma$ is the closure of a bounded convex domain in $H$, and every point of $\Gamma\setminus\{0\}$ is a positive multiple of a unique point in $H\cap\Gamma$.
%$\Gamma$ is said to be \emph{proper} if $L_\Gamma=\{0\}$. This is equivalent to any of the following conditions:
%\begin{itemize}
%	\item $\Gamma$ does not contain any entire straight line.
%	\item There exists an $(n-1)$-dimensional linear subspace $L\subset\R^n$ such that $L\cap\Gamma=\{0\}$.
%	\item There exists an affine hyperplane $H\subset\R^n$ not passing through $0$ such that $H\cap\Gamma$ is nonempty and compact.
%\end{itemize}
%In this case, $H\cap\Gamma$ is the closure of a bounded convex domain in $H$, and every point of $\Gamma\setminus\{0\}$ is a positive multiple of a unique point in $H\cap\Gamma$.
\item 
At the other extreme, note that $\Gamma$ is the whole $\R^n$ when $\dim L_\Gamma=n$, and is a closed half-space when $\dim L_\Gamma=n-1$. 
\item 
Given a convex cone $\Gamma\subset\R^n$, after choosing a linear subspace $L^\perp$ of $\R^n$ complementary to the ridge $L=L_\Gamma$, we can write $\Gamma$ as a product: 
$$
\Gamma=L\times \Gamma_0\subset L\oplus L^\perp=\R^n,
$$ 
where $\Gamma_0$ is a proper convex cone in the linear space $L^\perp$.
\end{itemize}

However, in the context of Theorem \ref{thm_B} and Corollary \ref{cor_C} in the introduction, $\R^n$ is understood as an affine space rather than a linear space, hence the above definitions should be understood with respect to a choice of the origin (or equivalently, modulo a translation). In particular, given $\Omega$, $\Gamma$, $L$, and $x_0$ as in Theorem \ref{thm_B}, after a translation, we may assume that $x_0$ is the origin $0\in\R^n$, so that $L$ is a $d$-dimensional linear subspace of $\R^n$ and $\Gamma$ is a convex cone with ridge $L_\Gamma=L$ under the above definition. By applying a linear transformation, we may further assume that
\begin{itemize}
	\item $L=\R^d\times\{0\}\subset\R^d\times\R^{n-d}=\R^n$;
	\vspace{0.1cm}
	\item $\Gamma=\R^d\times \Gamma_0\subset\R^d\times\R^{n-d}=\R^n$, where $\Gamma_0$ is a proper convex cone in $\R^{n-d}$.
\end{itemize}
Such a splitting $\R^n=\R^d\times\R^{n-d}$ will be used throughout the rest of the paper: when working in the setting of Theorem \ref{thm_A}, we will also assume, by a change of coordinates, that the flat boundary piece $F$ lies in $L=\R^d\times\{0\}=\Aff(F)$. In any case, $\R^d$ is understood as the spatial factor along the ridge or the flat boundary piece, and $\R^{n-d}$ as the transverse factor. The latter is not canonical and depends on a choice of a linear subspace complementary to $L$, but this choice plays no role in the paper. Finally, we always denote a point in $\R^n$ by $x=(y,z)$ with $y\in\R^d$ and $z\in\R^{n-d}$ (see Figure \ref{fig_domains}).

\begin{remark}
The shape of the convex cone $\Gamma\cong\R^d\times\Gamma_0$, which is determined by the shape of its proper factor $\Gamma_0$, plays no role in our analysis. For example, when requiring $\Omega$ to be contained in $\Gamma$ as in Theorem \ref{thm_B}, by enlarging $\Gamma$ and applying a linear transformation, we may assume that $\Gamma_0$ is a proper convex cone of any familiar form, say $\R^{n-d}_{\geq0}$ or 
$\big\{z\in\R^{n-d}\,\big|\,z_{n-d}\geq |(z_1,\cdots,z_{n-d-1})|\big\}$.
The main purpose of introducing convex cones and ridges in this paper is to give a coordinate-free exposition.
%avoiding coordinate notation of the above type.
\end{remark}

\section{Infinite-slope result: convergence toward directional derivatives}\label{sec_infinite}

We prove Theorem \ref{thm_A} \ref{item_thmA1} in this section. For clarity, we first treat the $d=0$ case in \S \ref{subsec_d0} to illustrate the main idea, and then the case of general $0\leq d\leq n-1$ in \S \ref{subsec_generald}.

\subsection{Case $d=0$}\label{subsec_d0}
As a toy version of the main construction of this section, consider a $1$-dimensional convex function $u\in C([0,1])$ with $u(0)=0$. By ``zooming in'' near $0$, namely scaling both the domain and the function values by a factor $\lambda>1$, we get the function
$u_\lambda(x):=\lambda u(\lambda^{-1}x)$.
It is decreasing in $\lambda$ and converges,  if $u$ has finite slope at $0$, to the linear function $u_\infty(x)=u'(0)x$. See Figure \ref{fig_ulambda}

\begin{figure}[ht]
	\centering
	\includegraphics[width=0.37\textwidth]{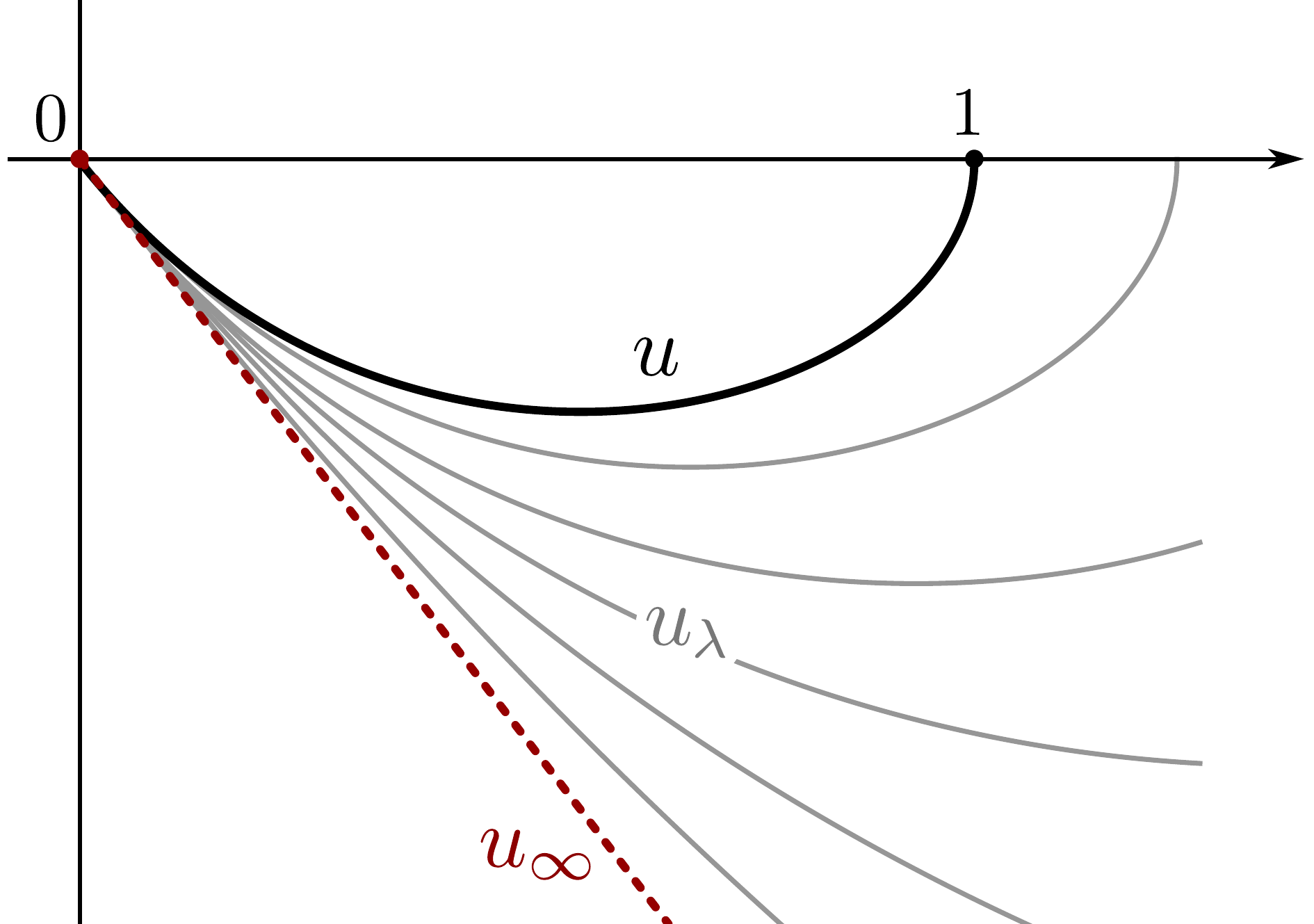}
	\caption{$u_\lambda\to u_\infty$ as $\lambda\to+\infty$.}
	\label{fig_ulambda}
\end{figure}

The same convergence holds in the more general setting where $u$ is defined on a bounded convex domain $\Omega\subset\R^n$ with $0\in\pa\Omega$ and $u(0)=0$; the only complication is that the limit function $u_\infty$ is now a $1$-homogeneous convex function on the \emph{tangent cone} of $\Omega$ at $0$. This function is extensively used in convex analysis and is called \emph{one-sided directional derivatives} (cf.\ \S \ref{subsec_MA} and \cite[\S 23]{Rockafellar}).
However, for our purposes, we do not need to specify the domain of $u_\infty$---knowing that $u_\lambda$ converges to $u_\infty$ within any sub-cone $\Gamma$ of the tangent cone suffices. We formulate the relevant part of the construction as the following proposition. Here and below, we let $B^m_r$ denote the open ball of radius $r$ in $\R^m$ centered at the origin, and omit the superscript if $m=n$. Note that by our convention, convex cones are closed sets with nonempty interior (see \S \ref{subsec_coordinates}).
\begin{proposition}[Directional derivatives at a boundary point]\label{prop_d0}
Let $\Omega\subset\R^n$ be a bounded convex domain, $u\in C(\overline{\Omega})$ be a convex function, and $\Gamma\subset\R^n$ be a proper convex cone. Assume that 
\begin{itemize}
\item The boundary $\pa\Omega$ contains the origin $0\in\R^n$.
\item $\Gamma\cap B_\epsilon \setminus\{0\}\subset \Omega$ for some $\epsilon>0$.
\item
$u(0)=0$.
\item
$u$ has finite slope at $0$.
\end{itemize}
Consider the family $(u_\lambda)_{\lambda\geq1}$ of convex functions defined by
\begin{equation}\label{eqn_propd0}
u_\lambda:\Gamma\cap B_{\lambda\epsilon}\to\R,\quad u_\lambda(x):=\lambda u(\lambda^{-1}x).
\end{equation}
Then there exists a $1$-homogeneous continuous convex function $u_\infty$ on $\Gamma$ such that $u_\lambda$ converges to $u_\infty$ uniformly on compact subsets of $\Gamma$ as $\lambda\to+\infty$.
\end{proposition}
\begin{remark}[Precise meaning of the uniform convergence]\label{rk_domain}
	The domain of $u_\lambda$ expands as $\lambda$ increases, and eventually contains any compact subset of $\Gamma$. The statement ``$u_\lambda$ converges uniformly on compact subsets of $\Gamma$'' should be understood as ``for any compact set $K\subset\Gamma$, there exists $\lambda_0\geq1$ such that $(u_\lambda)_{\lambda\geq\lambda_0}$ is defined on $K$ and converges uniformly on $K$''.
\end{remark}

\begin{remark}[$1$-homogeneous convex functions on a proper convex cone]\label{rk_homogeneity}
By a ``$1$-homogeneous convex function on $\Gamma$'', we mean a convex function $f:\Gamma\to\R$ such that $f(tx)=tf(x)$ for all $t\geq0$. Given a hyperplane $H\subset\R^n\setminus\{0\}$ such that $H\cap\Gamma$ is nonempty and compact, such a function $f$ is determined by its restriction to $H\cap\Gamma$, which can be any convex function $H\cap\Gamma\to\R$. Requiring $f$ to be continuous is the same as requiring $f|_{H\cap \Gamma}$ to be continuous up to the boundary.  A basic fact which plays a crucial role below is that for any $1$-homogeneous convex function $f:\Gamma\to\R$, the Monge-Amp\`ere measure $\mathcal{M}f$ (defined in the interior $\Gamma^\circ$ of $\Gamma$) vanishes. A proof goes as follows. Given a point $x\in\Gamma^\circ$ and a subgradient $a\in\pa f(x)\subset\R^n$, we have
\begin{equation}\label{equation_homogeneousMA}
f(x')\geq f(x)+a\cdot (x'-x)\ \ \text{ for all }x'\in\Gamma.
\end{equation}
By taking $x'=tx$ with $0<t<1$ and $t>1$ respectively, we infer that
\begin{equation}\label{equation_homogeneousMA2}
f(x)=a\cdot x.
\end{equation}
Inserting \eqref{equation_homogeneousMA2} back into \eqref{equation_homogeneousMA}, we get
$$
f(x')\geq a\cdot x'\ \ \text{ for all }x'\in\Gamma,
$$
which means that $a$ is also a subgradient of $f$ at the origin $0$, i.e., $a$ is contained in
$$
\pa f(0)=\big\{b\in\R^n\,\big|\,f(y)\geq b\cdot y \text{ for all }y\in\Gamma\big\}=\bigcap_{y\in\Gamma}\big\{b\in\R^n\,\big|\,f(y)\geq b\cdot y\big\}.
$$
In the last expression, we wrote the convex set $\pa f(0)$ as an intersection of affine half-spaces. By \eqref{equation_homogeneousMA2}, $a$ lies on the boundary of one of these half-spaces, hence also lies on the boundary of $\pa f(0)$. Since the boundary of a convex set in $\R^n$ has zero Lebesgue measure, we conclude that $\mathcal{M}f$ is the zero measure, as required.
\end{remark}

\begin{proof}[Proof of Proposition \ref{prop_d0}]
	By the convexity of $u$, for any fixed $x\in\Gamma$, the value $u_\lambda(x)$ (well defined when $\lambda$ is large enough) is non-increasing in $\lambda$.
%	 Indeed, for $\lambda_1<\lambda_2$, we have
%	$$
%	u\left(\frac{x}{\lambda_2}\right)=u\left(\left(1-\frac{\lambda_1}{\lambda_2}\right)0+\frac{\lambda_1}{\lambda_2}\frac{x}{\lambda_1}\right)\leq \left(1-\frac{\lambda_1}{\lambda_2}\right)u(0)+\frac{\lambda_1}{\lambda_2}u\left(\frac{x}{\lambda_1}\right)=\frac{\lambda_1}{\lambda_2}u\left(\frac{x}{\lambda_1}\right).
%	$$
	Since $u$ has finite slope at $0$, the pointwise limit $u_\infty:=\lim_{\lambda\to+\infty}u_\lambda$ exists (cf.\ Definition \ref{def_slopes}).	
	Its $1$-homogeneity follows from the definition of $u_\lambda$.
	To show that the convergence is uniform on compact subsets, first note that 
	since $u_\lambda(tx)-u_\infty(tx)$ is a nonnegative convex function of $t\in[0,1]$ and vanishes at $t=0$, it is non-decreasing in $t$, hence
    $$
    0\leq u_{\lambda}(t x)-u_\infty(t x)\leq u_\lambda(x)-u_\infty(x)\ \ \text{ for all }t\in[0,1].
    $$	
  It follows that if the convergence is uniform on $E\subset\Gamma$, so is it on $\{tx\mid t\in[0,1],\, x\in E\}$. Therefore, we only need to show that the convergence is uniform on any compact set $K\subset\Gamma$ not containing $0$.

   Next, by considering a hyperplane $H\subset\R^n$ whose intersection with $\Gamma$ is compact and contained in $\Omega$, we may find a larger proper convex cone $\Gamma'\subset\R^n$ which contains $\Gamma\setminus\{0\}$ in its interior and still satisfies $\Gamma'\cap B_\epsilon\setminus\{0\}\subset\Omega$ for some $\epsilon>0$. Then the pointwise convergence $u_\lambda\to u_\infty$ also holds on $\Gamma'$.
   By Lemma \ref{lemma_uniform}, $u_\infty$ is convex (hence continuous) in the interior $\Gamma'^\circ$ of $\Gamma'$ and the convergence is uniform on compact subsets of $\Gamma'^\circ$. In particular, it is uniform on the above $K$, as required.
\end{proof}

This proposition is the $d=0$ case of Proposition  \ref{prop_generald} below.
%so we do not provide a separate proof here. 
%In the latter proposition, we will generalize the construction from a single boundary point to a flat boundary piece of any dimension $0\leq d\leq n-1$. 
Although that proposition will allow us to fully prove Theorem \ref{thm_A} \ref{item_thmA1}, we first present in this subsection the proof for $d=0$, as this case captures the main idea with simpler notation and technical details. The argument uses Lemma \ref{lemma_twodistances} below, whose proof in turn relies on Lemma \ref{lemma_twocones}.
\begin{lemma}\label{lemma_twodistances}
	Given $\Omega$ and $\Gamma$ as in Proposition \ref{prop_d0}, there exist constants $c,\epsilon>0$ such that 
	$$
	c|x|\leq \dist(x,\pa\Omega)\leq |x|\quad \text{ for all } x\in \Gamma\cap B_\epsilon\setminus\{0\}\subset\Omega.
	$$
\end{lemma}
\begin{proof}
	Since $0\in\pa\Omega$ and $|x|=\dist(0,x)$, the second ``$\leq$'' follows immediately from the definition of $\dist(x,\pa\Omega)$. We proceed to show the first ``$\leq$''. As in the proof of Proposition \ref{prop_d0}, we may slightly enlarge the cone $\Gamma$ to obtain a proper convex cone $\Gamma'\subset\R^n$ whose interior contains $\Gamma\setminus\{0\}$, such that $\Gamma'\cap B_\epsilon\setminus\{0\}\subset\Omega$ for some $\epsilon$. 
	After taking a smaller $\epsilon$ if necessary, it holds that 
	$$
	\dist(x,\pa\Gamma')\leq\dist(x,\pa\Omega)\ \ \text{ for all }x\in\Gamma'\cap B_\epsilon\setminus\{0\}\subset\Omega.
	$$
	In view of the next lemma, we obtain the required inequality. 
\end{proof}
\begin{lemma}\label{lemma_twocones}
	Let $\Gamma$ and $\Gamma'$ be convex cones in $\R^n$ such that  $\Gamma$ is proper and $\Gamma\setminus\{0\}$ is contained in the interior of $\Gamma'$. Then there exists a constant $c>0$ such that
	$\dist(x,\pa\Gamma')\geq c|x|$ for all $x\in\Gamma$.
\end{lemma}
\begin{proof}
	$\dist(x,\pa\Gamma')$ and $|x|$ are both $1$-homogeneous continuous functions of $x\in\Gamma'$, so the ratio $\dist(x,\pa\Gamma')/|x|$ is continuous in $x\in\Gamma'\setminus\{0\}$ and invariant under scaling. Consider an affine hyperplane $H\subset\R^n\setminus\{0\}$ such that $H\cap\Gamma$ is nonempty and compact. The ratio is positive for $x\in H\cap \Gamma$, hence has a positive lower bound $c$ on $H\cap\Gamma$. That is, the required inequality holds for $x\in H\cap\Gamma$. Since any point of $\Gamma\setminus\{0\}$ is a positive multiple of a point in $H\cap \Gamma$, it holds for all $x\in\Gamma$.
\end{proof}

\begin{proof}[Proof of Theorem \ref{thm_A} \ref{item_thmA1} for $d=0$]
By translation, we may assume that the boundary point ($0$-dimensional flat boundary piece) in question is the origin $0$. Given $\alpha\leq 2d-n=-n$ and a convex function $u\in C(\overline{\Omega})$ which vanishes at $0\in\pa\Omega$ and satisfies one of the three conditions
\begin{align}
&\det\D^2u(x)\geq M\,\dist(x,0)^\alpha=M|x|^\alpha\ \text{ for all }x\in\Omega,\label{eqn_infiniteproof1}\\
&\det\D^2u(x)\geq M\dist(x,\pa\Omega)^\alpha\ \text{ for all }x\in\Omega,\label{eqn_infiniteproof2}\\
&u(x)<0\ \text{ and }\ \det \D^2u\geq M(-u(x))^\alpha\ \text{ for all }x\in\Omega, \label{eqn_infiniteproof3}
\end{align}
we shall show that $u$ has infinite slope at $0$. 

Suppose by contradiction that $u$ has finite slope at $0$. Then we may pick a proper convex cone $\Gamma\subset\R^n$ such that the assumptions of Proposition \ref{prop_d0} hold. The proposition and the stability of Monge-Amp\`ere measure (Lemma \ref{lemma_stability}) yield the following weak convergence of measures\footnote{This convergence should be understood in a sense similar to Remark \ref{rk_domain}, i.e., for any Borel set $B\subset\R^n$ such that $\overline{B}$ is compact and contained in $\Gamma^\circ$, the domain of definition of $\mathcal{M}u_\lambda$ contains $B$ when $\lambda$ is sufficiently large, and $\lim_{\lambda\to+\infty}\mathcal{M}u_\lambda(B)=0$.} 
in the interior $\Gamma^\circ$ of $\Gamma$
\begin{equation}\label{eqn_infiniteproof4}
\mathcal{M}u_\lambda  \rightharpoonup \mathcal{M}u_\infty=0\ \ \text{ as }\lambda\to+\infty,
\end{equation}
where $\mathcal{M}u_\infty$ is the zero measure on $\Gamma^\circ$ because $u_\infty$ is $1$-homogeneous (see Remark \ref{rk_homogeneity}).

We claim that under any of the conditions \eqref{eqn_infiniteproof1}--\eqref{eqn_infiniteproof3}, the density $\det\D^2u_\lambda$ of $\mathcal{M}u_\lambda$ is bounded from below by a positive function not depending on $\lambda$. This contradicts \eqref{eqn_infiniteproof4} and completes the proof.

To prove the claim, first note that
$$
\det\D^2u_\lambda(x)=\lambda^{-2n}\lambda^n(\det\D^2u)(\lambda^{-1}x)=\lambda^{-n}(\det\D^2u)(\lambda^{-1}x),
$$
where the factors $\lambda^{-2n}$ and  $\lambda^n$ come from the scalings of the variable and function values, respectively\footnote{
Strictly speaking, the equalities make sense only when $u$ is $C^2$; otherwise, a suitable reinterpretation via measures is required. We leave this reinterpretation to the reader.
}, in the definition \eqref{eqn_propd0} of $u_\lambda$. 

Now, under condition \eqref{eqn_infiniteproof1}, since $\alpha\leq -n$, we have
\begin{equation}\label{eqn_proofpropd0.5}
\det\D^2u_\lambda(x)=\lambda^{-n}(\det\D^2u)(\lambda^{-1}x)\geq \lambda^{-n}M|\lambda^{-1}x|^{\alpha}=M\lambda^{-n-\alpha}|x|^\alpha\geq M|x|^\alpha
\end{equation}
for any $x\in\Gamma$ and sufficiently large $\lambda\geq1$, as required.

Under condition \eqref{eqn_infiniteproof2}, estimate \eqref{eqn_proofpropd0.5} is still valid and yields the desired conclusion.
Indeed, this estimate only requires \eqref{eqn_infiniteproof1} to hold for all $x$ in a region of the form $\Gamma\cap B_\epsilon\setminus\{0\}\subset\Omega$ rather than on the whole $\Omega$, whereas Lemma \ref{lemma_twodistances} implies that the restrictions of \eqref{eqn_infiniteproof1} and \eqref{eqn_infiniteproof2} to such regions are equivalent to each other (more precisely, there exist $M,\epsilon>0$ such that \eqref{eqn_infiniteproof1} holds on $\Gamma\cap B_\epsilon\setminus\{0\}\subset\Omega$ if and only if the same is true for \eqref{eqn_infiniteproof2}, with possibly different $M$ and $\epsilon$).

Finally, under condition \eqref{eqn_infiniteproof3}, the situation is similar: we note that 
\begin{equation}\label{eqn_proofpropd02}
-C_1|x|\leq u(x)\leq -C_2|x|\ \ \text{ for all }x\in\Gamma\cap B_\epsilon\setminus\{0\}\subset\Omega
\end{equation}
for constants $C_1> C_2>0$, where the first ``$\leq$'' follows from the contradiction assumption 	``$u$ has finite slope at $0$'', which means there exists a subgradient $a\in\pa u(0)$, and hence $u(x)\geq a\cdot x\geq -|a|\,|x|$;
the second ``$\leq$'' follows from the condition $u<0$ in \eqref{eqn_infiniteproof3} and the convexity of $u$. As a consequence, the restriction of condition \eqref{eqn_infiniteproof3} to $\Gamma\cap B_\epsilon\setminus\{0\}$ is also equivalent to that of \eqref{eqn_infiniteproof1}, assuming that $u$ has finite slope at $0$. Therefore, we can use \eqref{eqn_proofpropd0.5} again to complete the proof of the claim.
\end{proof}

\begin{remark}[Shortcuts from the sign of $\alpha$]\label{rk_shortcut}
	The negative sign of $\alpha\leq -n$ in the above proof implies some obvious shortcuts:
	\begin{itemize}
		\item Condition \eqref{eqn_infiniteproof2} is plainly stronger than \eqref{eqn_infiniteproof1}, as it holds for all $x\in\Omega$ that $\dist(x,\pa\Omega)\leq \dist(x,0)$, and hence $\dist(x,\pa\Omega)^\alpha\geq |x|^\alpha$. Therefore, the use of Lemma \ref{lemma_twodistances} in the proof was unnecessary.
		\item Similarly, the second ``$\leq$'' in \eqref{eqn_proofpropd02} was unnecessary.
	\end{itemize}
	However, we chose not to take these shortcuts, so that the proof can be generalized directly to all $d$ in the next subsection, where $\alpha\leq 2d-n$ does not have a definite sign.
\end{remark}

\subsection{The case of general $d$}\label{subsec_generald} 
To generalize the previous section from a boundary point to a flat boundary piece $F\subset\pa\Omega$ of any dimension $0\leq d\leq n-1$, we shall dilate $u$ only in the directions transverse to $F$, and make use of the convergence of the dilated function $u_\lambda$ toward the transverse directional derivatives, as formulated in the following proposition.

\begin{proposition}[Directional derivatives transverse to a flat boundary piece]\label{prop_generald}
	Let $\Omega$ be a bounded convex domain in $\R^n=\R^d\times\R^{n-d}$, $u\in C(\overline{\Omega})$ be a convex function, and $\Gamma_0$ be a proper convex cone in $\R^{n-d}$. Assume that 
	\begin{itemize}
		\item The boundary $\pa\Omega$ contains $U\times\{0\}\subset\R^d\times\R^{n-d}$ for a convex domain $U$ in $\R^d$ containing the origin.
		\item There exists $\epsilon>0$ such that
		$$
		U\times \big(\Gamma_0\cap B_\epsilon^{n-d}\setminus\{0\}\big)\subset \Omega.
		$$
		\item
		$u=0$ on $U\times\{0\}$.
		\item
		$u$ has finite slope at $(0,0)$.
	\end{itemize}
Consider the family $(u_\lambda)_{\lambda\geq1}$ of convex functions defined by
\begin{equation}\label{eqn_ulambda2}
u_\lambda:U\times\big(\Gamma_0\cap B^{n-d}_{\lambda\epsilon}\big)\to\R,\quad u_\lambda(y,z):=\lambda u(y,\lambda^{-1}z).
\end{equation}
Then there exists a $1$-homogeneous continuous convex function $h$ on $\Gamma_0$ such that 
$u_\lambda$ converges to the function
$$
u_\infty: U\times \Gamma_0\to\R,\quad u_\infty(y,z)=h(z)
$$ uniformly on compact subsets of $U\times\Gamma_0$ as $\lambda\to+\infty$.
\end{proposition}
As in Proposition \ref{prop_d0}, the function $u_\lambda$ is defined on a domain that expands as $\lambda$ increases; its uniform convergence on compact subsets should be understood in a sense similar to Remark \ref{rk_domain}.
\begin{proof}
The arguments in the proof of Proposition \ref{prop_d0} can be generalized to show that $u_\lambda\to u_\infty$ uniformly on compact subsets of $U\times \Gamma_0$, where $u_\infty(y,z)$ is a continuous convex function of $(y,z)\in U\times\Gamma_0$ which is $1$-homogeneous in $z$. We omit the details. The proof is completed by noting that any such function is independent of $y$, namely $u_\infty(y_1,z)=u_\infty(y_2,z)$ for all $y_1,y_2\in U$ and $z\in\Gamma_0$. To verify this, we use the convexity along the line segment joining $(y_1,0)$ and $(y_2,z)$ to get
	$$
	u_\infty((1-t)y_1+ty_2,tz)\leq (1-t)u_\infty(y_1,0)+t\hspace{1pt}u_\infty(y_2,z)
	$$
	for $t\in(0,1)$. Since $u_\infty(y,z)$ is $1$-homogeneous in $z$ (in particular, $u_\infty(y,0)=0$), by dividing both sides by $t$, it follows that
	$u_\infty((1-t)y_1+ty_2,z)\leq u_\infty(y_2,z)$.
	Taking the limit $t\to0^+$, we get $u_\infty(y_1,z)\leq u_\infty(y_2,z)$. The desired assertion follows by switching $y_1$ and $y_2$.
\end{proof}

\begin{figure}[ht]
	\centering
	\includegraphics[width=0.44\textwidth]{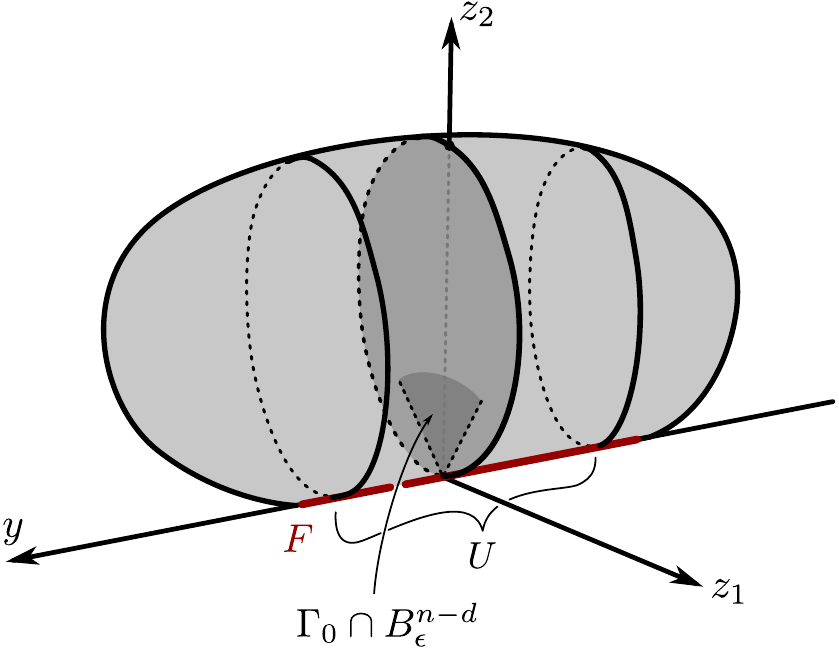}
	\caption{The setup for the proof of Theorem \ref{thm_A} \ref{item_thmA1}.}
	\label{fig_infiniteslopeproof}
\end{figure}

\begin{proof}[Proof of Theorem \ref{thm_A} \ref{item_thmA1} for general $d$]  
After a change of coordinates, we may assume that $L=\R^d\times\{0\}$ and that the flat boundary piece $F=\pa\Omega\cap L$ contains $U'\times\{0\}$ for a neighborhood $U'$ of the origin in $\R^d$. We claim that there exist a smaller neighborhood $U\subset U'$ of the origin in $\R^d$, a proper convex cone $\Gamma_0$ in $\R^{n-d}$, and $\epsilon>0$, such that 
$$
U\times(\Gamma_0\cap B^{n-d}_\epsilon\setminus\{0\})\subset\Omega.
$$ 
See Figure \ref{fig_infiniteslopeproof}. Indeed, the $d=0$ case of the claim  amounts to the obvious fact that if $0\in\pa\Omega$, then $\Omega$ contains $\Gamma\cap B_\epsilon\setminus\{0\}$ for some proper convex cone $\Gamma\subset\R^n$ and some $\epsilon>0$. For $d>0$, we first apply this fact to the convex domain $\Omega\cap\big(\{0\}\times\R^{n-d}\big)$ in $\{0\}\times\R^{n-d}$, obtaining a proper convex cone $\Gamma_0\subset\R^{n-d}$ such that the slice $\{0\}\times\big(\Gamma_0\cap B^{n-d}_\epsilon\setminus\{0\}\big)$ is contained in $\Omega$. Then we can use the compactness and the convexity to expand this slice into a region $U\times(\Gamma_0\cap B^{n-d}_\epsilon\setminus\{0\})$ which is still contained in $\Omega$, establishing the claim.

The rest of the proof is a straightforward generalization of the $d=0$ case in \S \ref{subsec_d0}, so we only give an outline. Suppose by contradiction that $u$ has finite slope at $F$. Then Proposition \ref{prop_generald} applies and yields the weak convergence 
$$
	\mathcal{M}u_\lambda  \rightharpoonup \mathcal{M}u_\infty=0\ \ \text{ as }\lambda\to+\infty.
$$
Here, $\mathcal{M}u_\infty$ is the zero measure on $U\times\Gamma_0$ like in the $d=0$ case because when $d>0$, the subgradients of $u_\infty(y,z)=h(z)$ are contained in the null set $\{0\}\times\R^{n-d}\subset\R^n$.  Again, we shall reach a contradiction by showing that if any of the conditions
\begin{align}
	&\det\D^2u(y,z)\geq M|z|^\alpha,\label{eqn_generaldproof1}\\
	&\det\D^2u\geq M\dist(\,\cdot\,,\pa\Omega)^\alpha,\label{eqn_generaldproof2}\\
	&u<0\ \text{ and }\ \det \D^2u\geq M(-u)^\alpha\label{eqn_generaldproof3}
\end{align}
holds in $U\times\big(\Gamma_0\cap B_\epsilon^{n-d}\setminus\{0\}\big)\subset\Omega$ with $\alpha\leq 2d-n$, then $\det\D^2u_\lambda$ is bounded from below by a positive function not depending on $\lambda$. 

To this end, note that by the definition \eqref{eqn_ulambda2} of $u_\lambda$, we have
$$
\det\D^2u_\lambda(y,z)=\lambda^{-2(n-d)}\lambda^n(\det\D^2u)(y,\lambda^{-1}z)=\lambda^{2d-n}(\det\D^2u)(y,\lambda^{-1}z).
$$
Therefore, condition \eqref{eqn_generaldproof1} implies 
$$
\det\D^2u_\lambda(y,z)=\lambda^{2d-n}(\det\D^2u)(y,\lambda^{-1}z)\geq \lambda^{2d-n}M|\lambda^{-1}z|^{\alpha}=M\lambda^{2d-n-\alpha}|z|^\alpha\geq M|z|^\alpha,
$$
as required. This estimate remains valid under either condition  \eqref{eqn_generaldproof2} or condition \eqref{eqn_generaldproof3}: the former is because, as a generalization of Lemma \ref{lemma_twodistances}, there exist $c,\epsilon>0$ such that
$$
c|z|\leq \dist((y,z),\pa\Omega)\leq |z|\ \ \text{ for all } (y,z)\in U\times\big(\Gamma_0\cap B_\epsilon^{n-d}\setminus\{0\}\big)\subset\Omega;
$$
the latter is because, as a generalization of \eqref{eqn_proofpropd02}, there exist $C_1>C_2>0$ such that 
$$
-C_1|z|\leq u(y,z)\leq -C_2|z|\ \ \text{ for all } (y,z)\in U\times\big(\Gamma_0\cap B_\epsilon^{n-d}\setminus\{0\}\big)\subset\Omega.
$$
This completes the proof.
\end{proof}

\section{Non-existence result: subgradient estimate}\label{sec_nonexistence}
 
We prove Theorem \ref{thm_A} \ref{item_thmA2} in this section. The proof is  based on the following lemma about the Monge-Amp\`ere measure of a convex function defined on a convex cone near the origin.
Recall that $B^m_r$ denotes the open ball of radius $r$ in $\R^m$ centered at the origin.
\begin{lemma}\label{lemma_super}
	Let $\Gamma_0$ be a proper convex cone in $\R^{n-d}$ and $u$ be a continuous convex function defined on the region
	$$
	D:=B^d_1\times ( \Gamma_0\cap B^{n-d}_1)\subset\R^d\times\R^{n-d}=\R^n,
	$$ 
    which is a neighborhood of the origin in the product cone $\Gamma=\R^d\times\Gamma_0$.  Let $K$ be a compact set in the interior $D^\circ=B^d_1\times(\Gamma_0^\circ\cap B^{n-d}_1)$ of $D$. Then, when $K$ undergoes the transformations
    $$
    \sigma_t:\R^d\times\R^{n-d}\to\R^d\times\R^{n-d},\quad \sigma_t(y,z):=(y,tz),\ \ 0<t<1,
    $$
   its mass with respect to the Monge-Amp\`ere measure $\mathcal{M}u$ of $u$ satisfies
	$$
	\mathcal{M}u(\sigma_t(K))=o(t^{d-n})\ \ \text{ as } t\to0^+.
	$$
\end{lemma}

\begin{figure}[ht]
	\centering
	\includegraphics[width=1.07\textwidth]{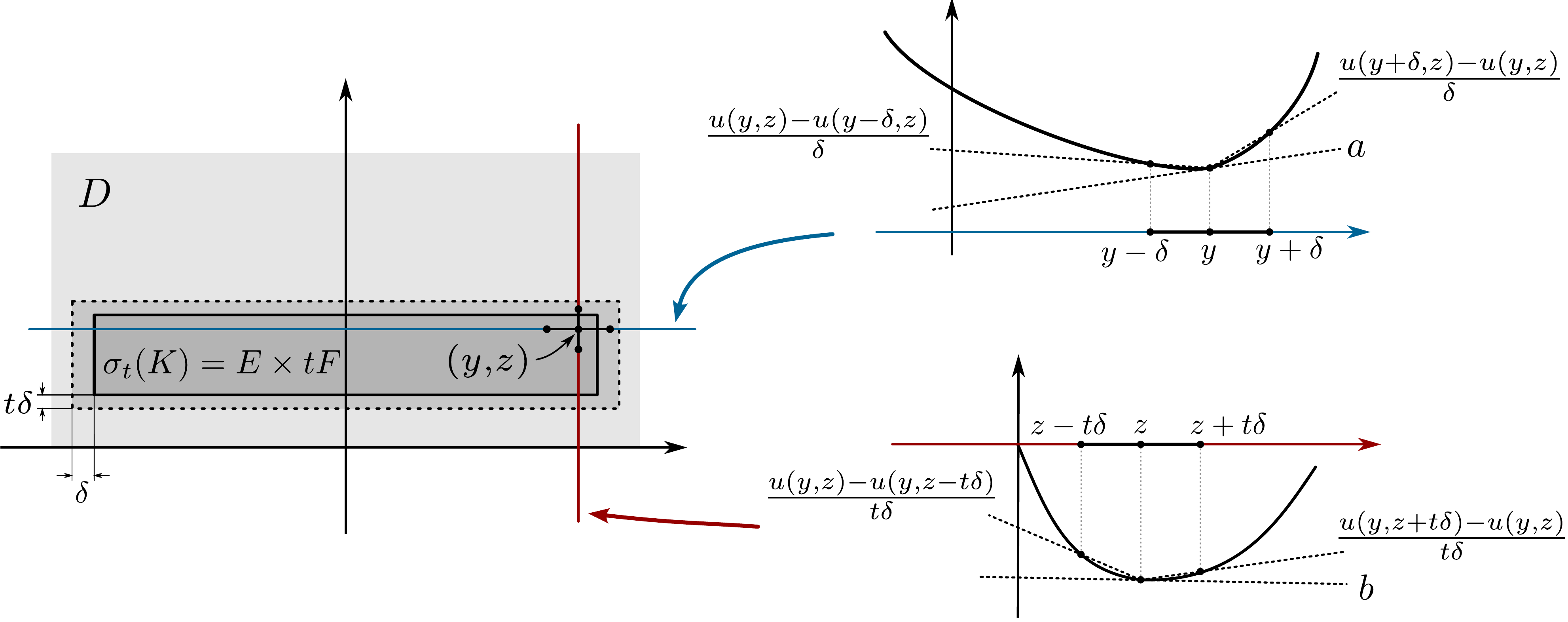}
	\caption{Illustration of the proof of Lemma \ref{lemma_super} when $d=n-d=1$ and $\Gamma_0=\R_{\geq0}$. Shown on the right are the section views of the graph of $u$ along the horizontal and vertical lines passing through $(y,z)\in\sigma_t(K)=E\times tF$. The quantities marked beside the dashed lines are the slopes. The argument consists in bounding each of the components $a$ and $b$ of a subgradient $(a,b)\in \pa u(y,z)$ by two slopes.}
	\label{fig_subgradient}
\end{figure}
\begin{proof}[Proof]
Since enlarging $K$ does not affect the statement, we may assume that $K$ has a product form
$$
K=E\times F,\quad E\subset B_1^d,\ \  F\subset\Gamma_0^\circ\cap B_1^{n-d},
$$
with $E$ and $F$ compact. By compactness, we may pick a sufficiently small $\delta>0$ such that
\begin{equation}\label{eqn_subgradientproof0}
\widetilde{E}:=\big\{y\in\R^d\,\big|\,\dist(y,E)\leq \delta\}\subset B_{1-\delta}^d,\quad \widetilde{F}:=\big\{z\in\R^{n-d}\,\big|\,\dist(z,F)\leq \delta\}\subset\Gamma_0^\circ\cap B_{1-\delta}^{n-d}.
\end{equation}
Note that the last inclusion implies 
\begin{equation}\label{eqn_subgradientproof0.5}
t\widetilde{F}\subset \Gamma_0^\circ\cap B^{n-d}_{t(1-\delta)} \quad \text{ for all }0<t<1.
\end{equation}

Given a point $(y,z)\in \sigma_t(K)=E\times tF$, the subgradients $(a,b)\in\pa u(y,z)$ of $u$ at $(y,z)$ are characterized by the condition
\begin{equation}\label{eqn_subgradientproof1}
	u(y',z')\geq u(y,z)+a\cdot (y'-y)+b\cdot(z'-z)\ \  \text{ for all } (y',z')\in D.
\end{equation}
We proceed to estimate each component of $a=(a_1,\cdots,a_d)$ and $b=(b_1,\cdots,b_{n-d})$. Let
$$
e_1=(1,0,\cdots,0),\cdots, e_d=(0,\cdots,0,1)\in\R^d,\quad f_1=(1,0,\cdots,0),\cdots, f_{n-d}=(0,\cdots,0,1)\in\R^{n-d}
$$
be the standard bases of $\R^d$ and $\R^{n-d}$, respectively. Then 
\begin{itemize}
	\item Taking 
	$(y',z')=(y\pm\delta e_i,z)$ in \eqref{eqn_subgradientproof1}, we obtain
	$$
	\frac{u(y,z)-u(y-\delta e_i,z)}{\delta} \leq a_i\leq \frac{u(y+\delta e_i,z)-u(y,z)}{\delta}.
	$$
	Since $y\in E\subset\widetilde{E}$, $y\pm \delta e_i\in\widetilde{E}$ and $z\in tF\subset t\widetilde{F}$, in view of \eqref{eqn_subgradientproof0} and \eqref{eqn_subgradientproof0.5}, we deduce that
	\begin{equation}\label{eqn_subgradientproof2}
		|a_i|\leq \frac{C}{\delta},\quad i=1,\cdots,d\,,
	\end{equation}
	where $C$ is twice the $L^\infty$-norm of $u$ on $\overline{B}_{1-\delta}^d\times\big(\Gamma_0\cap\overline{B}_{1-\delta}^{n-d}\big)$.
	\item 
	Taking $(y',z')=(y,z\pm t\delta f_j)$ in \eqref{eqn_subgradientproof1},  we obtain
	$$
	\frac{u(y,z)-u(y,z-t\delta f_j)}{t\delta} \leq b_j\leq \frac{u(y,z+t\delta f_j)-u(y,z)}{t\delta}.
	$$
	Since $y\in E\subset\widetilde{E}$, $z\in tF\subset t\widetilde{F}$ and $z\pm t\delta f_j\in t\widetilde{F}$, in view of \eqref{eqn_subgradientproof0} and \eqref{eqn_subgradientproof0.5}, we deduce that
	\begin{equation}\label{eqn_subgradientproof3}
		|b_j|\leq \frac{c(t)}{t\delta},\quad j=1,\cdots,n-d\,,
	\end{equation}
	where
	$$
	c(t):=2\sup_{y\in \overline{B}_{1-\delta}^d,\, z\in\Gamma_0\cap\overline{B}_{t(1-\delta)}^{n-d}} |u(y,z)-u(y,0)|.
	$$
	Note that $c(t)\to 0$ as $t\to 0^+$: otherwise,  there exists a sequence $(y^{(k)},z^{(k)})_{k=1,2,\cdots}$ in $\overline{B}_{1-\delta}^d\times\Gamma_0$ such that $\lim_{k\to\infty}|z^{(k)}|=0$ whereas $|u(y^{(k)},z^{(k)})-u(y^{(k)},0)|$ is larger than a positive constant; then picking a subsequence of $(y^{(k)})$ converging to some $y^{(\infty)}\in\overline{B}_{1-\delta}^d$ leads to a contradiction with the continuity of $u$ at $(y^{(\infty)},0)$.
%		Note that $c(t)\to 0$ as $t\to 0^+$: otherwise,  there exists a sequence 
%	$$
%	(y^{(k)},z^{(k)})\in\overline{B}_{1-\delta}^d\times \Gamma_0,\quad k=1,2,\cdots
%	$$ 
%	such that 
%	$$
%	\lim_{k\to\infty}|z^{(k)}|=0,\quad\  |u(y^{(k)},z^{(k)})-u(y^{(k)},0)|>c_0>0\ \text{ for } k=1,2,\cdots;
%	$$ 
%	picking a subsequence of $y^{(k)}$ converging to some $y^{(\infty)}\in\overline{B}_{1-\delta}^d$ leads to a contradiction with the continuity of $u$ at $(y^{(\infty)},0)$.
\end{itemize}

Since $(y,z)\in\sigma_t(K)$ is arbitrary, estimates \eqref{eqn_subgradientproof2} and \eqref{eqn_subgradientproof3} imply
$$
\pa u\big(\sigma_t(K)\big)\subset \left[-\frac{C}{\delta}, \frac{C}{\delta}\right]^d\times\left[-\frac{c(t)}{t\delta},\frac{c(t)}{t\delta }\right]^{n-d}.
$$
Since $\mathcal{M}u(\sigma_t(K))$ is defined to be the volume of $\pa u\big(\sigma_t(K)\big)$ (see \S \ref{subsec_MA}), the desired conclusion follows.
\end{proof}

\begin{proof}[Proof of Theorem \ref{thm_A} \ref{item_thmA2}]
As explained in the first paragraph of the proof of Theorem \ref{thm_A} \ref{item_thmA1} in \S \ref{subsec_generald}, we may assume that $L=\R^d\times\{0\}$, that the flat boundary piece $F=\pa\Omega\cap L$ contains $U\times\{0\}$ for a neighborhood $U$ of the origin in $\R^d$, and that $\Omega$ contains $U\times(\Gamma_0\cap B^{n-d}_\epsilon\setminus\{0\})$ for a proper convex cone $\Gamma_0$ in $\R^{n-d}$ and some $\epsilon>0$.
By further rescaling the coordinates, we may assume that $\overline{\Omega}$ contains $D:=B^d_1\times(\Gamma_0\cap B^{n-d}_1)$.

Suppose by contradiction that $u\in C(\overline{\Omega})$ is convex and satisfies 
\begin{equation}\label{eqn_superproof1}
\det\D^2u(y,z)\geq M\dist((y,z),L)^\alpha=M|z|^\alpha\ \ \text{ for all }(y,z)\in\Omega,
\end{equation}
with $\alpha\leq 2d-2n$. Let $K$ be a compact set contained in the interior $D^\circ$ of $D$. Condition \eqref{eqn_superproof1} implies 
$$
\mathcal{M}u(\sigma_t(K))=\int_{\sigma_t(K)}\det\D^2u(y,z)\dif y\,\dif z\geq M\int_{\sigma_t(K)}|z|^\alpha\dif y\,\dif z=Mt^{\alpha+n-d}\int_K |z|^\alpha \dif y\,\dif z,
$$
where $\sigma_t(y,z):=(y,tz)$, $0<t<1$. We may assume that $K$ has nonempty interior, so that the last integral is positive. Since $\alpha+n-d\leq d-n$, 
this inequality contradicts the estimate $\mathcal{M}u(\sigma_t(K))=o(t^{d-n})$ provided by Lemma \ref{lemma_super},  completing the proof.
\end{proof}

\section{Finite-slope and existence results: explicit barriers}\label{sec_finiteslope}

We prove Theorem \ref{thm_B} in this section. The proofs of parts \ref{item_thmB1} and \ref{item_thmB2} are given in \S \ref{subsec_finiteslopeA} and \S \ref{subsec_existence}, respectively, both based on the explicit barriers (subsolutions) constructed in \S \ref{subsec_barriers}. Note that, although we assume $0\leq d\leq n-1$ throughout the paper, some expressions in this section require $d\geq 1$ to make sense, but we omit a separate discussion of the case $d=0$ since it only requires a minor modification from $d\geq1$.

\subsection{Explicit subsolutions to $\det\D^2u(y,z)=|z|^\alpha$}\label{subsec_barriers}
As before, we denote points in $\R^n=\R^d\times\R^{n-d}$ by $(y,z)$. In Lemmas \ref{lemma_barrier} and \ref{lemma_barrier2} below, we exhibit two explicit convex subsolutions to the equation $\det\D^2u(y,z)= |z|^\alpha$ of the product form $(|y|^2\pm\Lambda)\Phi(z)$ with $\Phi(0)=0$. They will be used in \S \ref{subsec_finiteslopeA} and \S \ref{subsec_existence}, respectively.

To verify that they are indeed convex subsolutions, we calculate their Hessians using the following standard linear algebraic lemma. In what follows, all matrices are real, ``$A>0$'' means ``$A$ is positive definite'' for a matrix $A$, and $v\in\R^m$ is treated as a column vector in 
expressions like $v\transp{v}$ and $\transp{v}Av$. 

\begin{lemma}\label{lemma_LA}
	Let $A$ (resp.\ $C$) be a non-degenerate symmetric matrix of dimension $d_A$ (resp.\ $d_C$), $B$ be a $d_A\times d_C$ matrix, and let  $p,q\in\R$, $v\in\R^{d_A}$. Then we have the following formulas for determinants and criteria for positive definiteness:
\begin{equation}\label{eqn_det1}
	\det\begin{pmatrix}
		A&B\\
		\transp{B}&C
	\end{pmatrix}
	=\det(A)\det\left(C-\transp{B} A^{-1}B\right),
\end{equation}
\begin{equation}\label{eqn_pd1}
\begin{pmatrix}
	A&B\\
	\transp{B}&C
\end{pmatrix}>0
\ \Longleftrightarrow\ 
A>0\,\text{ and }\, C-\transp{B} A^{-1}B>0,
\end{equation}
\begin{equation}\label{eqn_det2}
%	\det(a A+b B\transp{B})=a^{d_A-d_C}\det(A)\det(aI_{d_C}+b\transp{B}A^{-1}B).
		\det(p A+q v\transp{v})=p^{d_A-1}(p+q\,\transp{v}A^{-1}v)\det(A).
\end{equation}
Moreover, when $p,A>0$, 
\begin{equation}\label{eqn_pd2}
	p A+q v\transp{v}>0 
	\ \Longleftrightarrow\
%	p+q\,\transp{v}A^{-1}v>0 
%	\,\text{ and }\, 
%	A>0.
\det(p A+q v\transp{v})>0.
\end{equation}
\end{lemma}

\begin{proof}
Equation \eqref{eqn_det1} follows from the congruences
\begin{equation}\label{eqn_congruence}
\begin{pmatrix}
	A&B\\
	\transp{B}&C
\end{pmatrix}
\sim
\begin{pmatrix}
	A&0\\
	0&C-\transp{B}A^{-1}B
\end{pmatrix}
\sim
\begin{pmatrix}
	A-BC^{-1}\transp{B}&0\\
	0&C
\end{pmatrix},
\end{equation}
which also yield
$$
\det(A)\det\left(C-\transp{B} A^{-1}B\right)=\det(C)\det\left(A-B C^{-1}\transp{B}\right).
$$
Equation \eqref{eqn_det2} is trivial if $p=0$ or $q=0$, and otherwise follows from the above equality by letting $C=-q^{-1}$, $B=v$ and rescaling $A$. The equivalence \eqref{eqn_pd1} follows from \eqref{eqn_congruence}. Finally, to show the equivalence \eqref{eqn_pd2}, we may assume $p=1$ and $A=I$ by multiplying $pA+qv\transp{v}$ both on the left and right by $(pA)^{-\frac{1}{2}}$ and replacing $v$ by $(pA)^{-\frac{1}{2}}v$. If $v=0$, the equivalence is trivial; otherwise, for $I+q v\transp{v}$ to be positive definite, it is sufficient that $\transp{x}\big(I+qv\transp{v}\big)x>0$ for $x=v$ and all $x\in v^\perp\setminus\{0\}$. But we have
$$
\transp{x}\big(I+qv\transp{v}\big)x=
\begin{cases}
	|v|^2+q|v|^4=|v|^2\det(I+qv\transp{v})&\text{ if }x=v\\
	p\transp{x}Ax>0 & \text{ if }x\in v^\perp\setminus\{0\}
\end{cases}
$$
by \eqref{eqn_det2}. The required equivalence follows.
\end{proof}

Our first subsolution, as given in the following lemma, is a slight generalization of the Pogorelov-type barrier constructed by Caffarelli \cite{Caffarelli_1993} (see also \cite{CaffarelliYuan2022, MooneyRakshit2023}), which in turn generalizes the original construction of Pogorelov for $d=n-1$.
%It requires $\alpha>2d-n$ and is defined on the slab $B_R^d\times \R^{n-d}$.
\begin{lemma}\label{lemma_barrier}
	Given $\alpha>2d-n$ and $M,R>0$, there exists $\Lambda>0$ such that the function 
	\begin{align*}
		f:\R^n=&\ \R^d\times\R^{n-d}\to \R,\\[-0.2cm]
		&f(y,z):=(|y|^2+\Lambda)|z|^\beta,\quad \beta:=\frac{\alpha+2n-2d}{n}>1
	\end{align*}
	satisfies 
	\begin{equation}\label{eqn_lemmabarrier1}
	\D^2 f(y,z)>0\ \text{ and }\ \det\D^2 f(y,z)\geq M|z|^\alpha \quad \text{if }  |y|<R\text{ and }z\neq 0.
	\end{equation}
\end{lemma}
\begin{proof}
	We will omit the variables $(y,z)$ in expressions like $f(y,z)$ and $\D^2f(y,z)$. The Hessian $\D^2f$ can be written in the block form
	\begin{equation}\label{eqn_fhessianblocks}
		\D^2 f=\begin{pmatrix}
			\pa^2_{yy}f&\pa^2_{yz}f\\[1pt]
			\pa^2_{zy}f&\pa^2_{zz}f
		\end{pmatrix},
	\end{equation}
where $\pa_{yy}^2f$ denotes the matrix $\big(\frac{\pa^2f}{\pa y_i\pa y_j}\big)_{1\leq i,j\leq d}$, and similarly for the other blocks. Computing the second derivatives, we obtain 
\begin{equation}\label{eqn_proofbarrier0}
	\begin{split}
	&\pa^2_{yy}f=2|z|^\beta I_d,\qquad \pa^2_{zz}f=\beta(|y|^2+\Lambda)|z|^{\beta-2}\left(I_{n-d}+\frac{\beta-2}{|z|^2}z\transp{z}\right),\\
	&\pa_{zy}^2f=2\beta|z|^{\beta-2} z\transp{y},\qquad \pa_{yz}^2f=2\beta|z|^{\beta-2} y \transp{z},
	\end{split}
\end{equation}
where $I_m$ denotes the $m\times m$ identity matrix. To apply Lemma \ref{lemma_LA}, note that $\pa^2_{yy}f>0$ for $z\neq0$, whereas
\begin{equation}\label{eqn_proofbarrier1}
	\begin{split}
	\pa^2_{zz}f-\pa^2_{zy}f(\pa_{yy}^2f)^{-1}\pa^2_{yz}f&=\beta|z|^{\beta-2}(|y|^2+\Lambda)\left(I_{n-d}+\frac{\beta-2}{|z|^2}z\transp{z}\right)-\frac{(2\beta|z|^{\beta-2})^2}{2|z|^\beta}|y|^2z\transp{z}\\
	&=\beta|z|^{\beta-2}\left[(|y|^2+\Lambda)I_{n-d}+\frac{(\beta-2)\Lambda-(\beta+2)|y|^2}{|z|^2}z\transp{z}\right]
	\end{split}
\end{equation}
has the form $pA+qv\transp{v}$ with $p,A>0$. Therefore, by \eqref{eqn_pd1} and \eqref{eqn_pd2}, the first inequality in \eqref{eqn_lemmabarrier1} follows from the second, hence it suffices to establish the latter. This amounts to a straightforward calculation: applying \eqref{eqn_det1} first, then \eqref{eqn_proofbarrier0} and \eqref{eqn_proofbarrier1}, and finally \eqref{eqn_det2}, we obtain
	\begin{align*}
		\det \D^2 f&=\det \pa^2_{yy}f\, \det\left(\pa^2_{zz}f-\pa^2_{zy}f(\pa_{yy}^2f)^{-1}\pa^2_{yz}f\right)\\
		&=(2|z|^\beta)^d (\beta|z|^{\beta-2})^{n-d}\det\left[(|y|^2+\Lambda)I_{n-d}+\frac{(\beta-2)\Lambda-(\beta+2)|y|^2}{|z|^2}z\transp{z}\right]\\
		&=(2|z|^\beta)^d (\beta|z|^{\beta-2})^{n-d}(|y|^2+\Lambda)^{n-d-1}\left(|y|^2+\Lambda+(\beta-2)\Lambda-(\beta+2)|y|^2\right)\\
%&=2^d\beta^{n-d}|z|^{n\beta-2(n-d)}(|y|^2+\Lambda)^{n-d-1}\left[(\beta-1)\Lambda-(\beta+1)|y|^2\right].
&=2^d\beta^{n-d}\left(\frac{\Lambda-|y|^2}{\Lambda+|y|^2}\beta-1\right)|z|^{n\beta-2(n-d)}(|y|^2+\Lambda)^{n-d}.
	\end{align*}
	Since $n\beta-2(n-d)=\alpha$ and $\beta>1$, we conclude that $\det \D^2 f\geq M|z|^\alpha$ for $|y|<R$ and $z\neq0$ if $\Lambda=\Lambda(n,d,\beta,M,R)$ is sufficiently large, as required.
\end{proof}

In the second subsolution, as given in the next lemma, the $z$-factor is changed from $|z|^\beta$ into $(z_1\cdots z_{n-d})^\gamma$, which is a concave function of $z=(z_1,\cdots,z_{n-d})\in\R^{n-d}_{\geq0}$ when $0<\gamma\leq\frac{1}{n-d}$, and is strictly concave for $z_1,\cdots,z_{n-d}>0$ and $\gamma<\frac{1}{n-d}$. Meanwhile, the $y$-factor is changed into the negative convex function $|y|^2-\Lambda$ ($|y|<R<\sqrt{\Lambda}$). 
\begin{lemma}\label{lemma_barrier2}
%	Given $\alpha>2d-2n$, pick
%	$$
%	0<\beta<\min\left\{1,\frac{2(n-d)}{n},\frac{\alpha+2(n-d)}{n}\right\}。
%	$$
%Then for any $M,R>0$, there exists $\Lambda>0$ such that the function 
%$$
%		f:\R^d\times\R^{n-d}_{\geq0}\to \R,\quad f(y,z):=(|y|^2-\Lambda)(z_1\cdots z_{n-d})^\frac{\beta}{n-d},
%$$
	Given $2d-2n<\alpha< 2d-n$ and $M,R>0$, there exists $\Lambda>R^2$ such that the function 
\begin{align*}
	f:&\R^d\times\R^{n-d}_{\geq0}\to \R,\\[-0.2cm]
	&f(y,z):=(|y|^2-\Lambda)(z_1\cdots z_{n-d})^\frac{\beta}{n-d},\quad \beta:=\frac{\alpha+2n-2d}{n}\in(0,1)
\end{align*}
	satisfies 
	\begin{equation}\label{eqn_lemmabarrier2}
	\D^2 f(y,z)>0\ \text{ and }\ \det\D^2 f(y,z)\geq M(z_1\cdots z_{n-d})^\frac{\alpha}{n-d} \quad \text{if }  |y|<R \text{ and } z_1,\cdots,z_{n-d}>0.
	\end{equation}
%	(The first inequality means positive definiteness of the Hessian $\D^2f$.)
\end{lemma}
Note that we have $|z_j|\leq |z|$ for $z\in\R^{n-d}_{\geq0}$, but the reverse control does not hold, hence the required inequality $\det\D^2 f\geq M|z|^\alpha$ holds throughout $B^d_R\times\R^{n-d}_{\geq0}$ only when $\alpha\leq0$. This issue can easily be fixed when we apply this barrier (see Lemma \ref{lemma_barrier3} below).
%This can easily be fixed in several ways, for example by restricting to a sub-cone of $\R^{n-d}_{\geq0}$ as in the proof of Lemma \ref{lemma_barrier3} below.

\begin{proof}
The proof is analogous to that of Lemma \ref{lemma_barrier}, so we only present the main calculations. Putting $\gamma:=\frac{\beta}{n-d}$ and $z^{-1}:=\transp{(z_1^{-1},\cdots,z_{n-d}^{-1})}$, we have
$$
\pa^2_{yy}f=2(z_1\cdots z_{n-d})^\gamma I_d,\qquad \pa^2_{zz}f=\gamma(|y|^2-\Lambda)(z_1\cdots z_{n-d})^\gamma\left[\gamma z^{-1}\transp{(z^{-1})}-\diag\big(z_1^{-2},\cdots,z_{n-d}^{-2}\big)\right],
$$
$$
\pa^2_{zy}f=2\gamma(z_1\cdots z_{n-d})^\gamma z^{-1}\transp{y},\qquad
\pa^2_{yz}f=2\gamma(z_1\cdots z_{n-d})^\gamma y\transp{(z^{-1})},
$$
\begin{align*}
	\pa^2_{zz}f-\pa^2_{zy}f(\pa_{yy}^2f)^{-1}\pa^2_{yz}f=\gamma(z_1\cdots z_{n-d})^\gamma\Big[
	(\Lambda-|y|^2)\diag\big(z_1^{-2},\cdots,z_{n-d}^{-2}\big)
	-\gamma(|y|^2+\Lambda)z^{-1}\transp{(z^{-1})}
	\Big],
\end{align*}
\begin{align*}
	&\det \D^2 f=\det \pa^2_{yy}f\, \det\left(\pa^2_{zz}f-\pa^2_{zy}f(\pa_{yy}^2f)^{-1}\pa^2_{yz}f\right)\\
	&=\big[2(z_1\cdots z_{n-d})^\gamma\big]^d\big[\gamma(z_1\cdots z_{n-d})^\gamma\big]^{n-d}\det\Big[
	(\Lambda-|y|^2)\diag\big(z_1^{-2},\cdots,z_{n-d}^{-2}\big)
	-\gamma(|y|^2+\Lambda)z^{-1}\transp{(z^{-1})}
	\Big]\\
	&=2^d\gamma^{n-d}(z_1\cdots z_{n-d})^{\gamma n}(\Lambda-|y|^2)^{n-d-1}(z_1\cdots z_{n-d})^{-2}\big[\Lambda-|y|^2-\gamma(n-d)(|y|^2+\Lambda)\big]\\
	&=2^d\gamma^{n-d}(\Lambda-|y|^2)^{n-d}\left(1-\frac{\Lambda+|y|^2}{\Lambda-|y|^2}\beta\right)(z_1\cdots z_{n-d})^\frac{\alpha}{n-d}.
\end{align*}
Since $\beta<1$, the last expression implies the second inequality in \eqref{eqn_lemmabarrier2}. 
By Lemma \ref{lemma_LA} and the expression of $\pa^2_{zz}f-\pa^2_{zy}f(\pa_{yy}^2f)^{-1}\pa^2_{yz}f$, the first inequality in \eqref{eqn_lemmabarrier2} follows.
\end{proof}

\subsection{The finite-slope result}\label{subsec_finiteslopeA}
%Let $\Omega\subset\R^n$ be a bounded convex domain, $F\subset\pa\Omega$ be a $d$-dimensional flat boundary piece, and $L:=\Aff(F)$ be the affine hull of $F$. 
We now adopt the assumptions of Theorem \ref{thm_B} and prove part \ref{item_thmB1} of the theorem. As discussed in \S \ref{subsec_coordinates}, we may fix a coordinate system 
$x=(y,z)$ on $\R^n=\R^d\times\R^{n-d}$ with $y\in \R^d$ and $z\in\R^{n-d}$, such that
$$
x_0=(0,0)\in\pa\Omega,\quad L=\R^d\times\{0\},\quad \Omega\subset\Gamma=\R^{d}\times\Gamma_0,
$$
where $\Gamma_0$ is a proper convex cone in $\R^{n-d}$.
Let $u\in C(\overline{\Omega})$ be a convex function vanishing on $\pa\Omega$ and satisfying either of the two conditions
\begin{equation}\label{eqn_finiteslopeproof1}
	\det\D^2u(y,z)\leq M\,\dist((y,z),L)^\alpha=M |z|^\alpha \ \ \text{ for all }(y,z)\in\Omega,
\end{equation}
\begin{equation}\label{eqn_finiteslopeproof1.1}
	\det\D^2u(y,z)\leq M(-u(y,z))^\alpha \ \ \text{ for all }(y,z)\in\Omega,
\end{equation}
with $\alpha>2d-n$ and $M>0$. We need to show that $u$ has finite slope at $x_0$. We  present the proofs under the two conditions separately.

\begin{proof}[Proof of Theorem \ref{thm_B} \ref{item_thmB1}, case \eqref{eqn_finiteslopeproof1}] 
Consider the function $f(y,z)$ produced by Lemma \ref{lemma_barrier}, where we take the parameter $R$ to be
$\max_{(y,z)\in\overline{\Omega}}|y|$ and
take the $M$ from condition \eqref{eqn_finiteslopeproof1}. Then $f$ is convex on $\overline{\Omega}$ and satisfies 
\begin{equation}\label{eqn_finiteslopeproof1.2}
	\det \D^2f(y,z)\geq M|z|^\alpha \ \ \text{ for all }(y,z)\in\Omega.
\end{equation}
Since the convex cone $\Gamma_0\subset\R^{n-d}$ is proper, we can choose $\xi\in\R^{n-d}$ with $|\xi|=1$ such that
%the orthogonal complement $\xi^\perp$ of $\xi$ in $\R^{n-d}$ intersects $\Gamma_0$ only at the origin, and that 
the half-space 
$$
\mathcal{H}_\xi:=\big\{z\in\R^{n-d}\,\big|\,\xi\cdot z\geq 0\big\}\subset\R^{n-d}
$$
contains $\Gamma_0\setminus\{0\}$ in its interior. We shall bound $u$ from below by a linear modification $v$ of $f$ of the form
\begin{equation}\label{eqn_finiteslopeproof0}
v:\overline{\Omega}\to\R,\quad v(y,z):=f(y,z)-C\,\xi\cdot z.
\end{equation}
In order to determine the coefficient $C>0$ of the modification term, we note that 
\begin{equation}\label{eqn_finiteslopeproof2}
C_1|z|\leq \xi\cdot z\leq |z| \ \ \text{ for all } z\in \Gamma_0,
\end{equation}
with a constant $C_1=C_1(\Gamma_0,\xi)\in(0,1)$. Here, the first ``$\leq$'' follows from Lemma \ref{lemma_twocones} by taking the cones $\Gamma$ and $\Gamma'$ therein to be $\Gamma_0$ and $\mathcal{H}_\xi$, respectively, and the second from Cauchy-Schwarz. Meanwhile, since $\Omega$ is bounded and $f(y,z)$ has the expression $(|y|^2+\Lambda)|z|^{\beta}$ with $\Lambda>0$ and $\beta>1$, we have
\begin{equation}\label{eqn_finiteslopeproof3}
	0\leq f(y,z)\leq C_2|z| \ \ \text{ for all } (y,z)\in\overline{\Omega},
\end{equation}
with a constant $C_2=C_2(\Lambda,\beta,\Omega)>0$. Then we may take
$C$ to be any number no less than $C_2/C_1$, so that \eqref{eqn_finiteslopeproof2} and \eqref{eqn_finiteslopeproof3} imply $v\leq0$ on $\overline{\Omega}$.

	By the assumptions on $u$ and the construction of $v$, we have
	$$
	u=0\geq v\ \text{ on }\pa\Omega,\quad \det\D^2u\leq \det \D^2 v \ \text{ in }\Omega.
	$$
	It follows that $u\geq v$ in $\Omega$ by the comparison principle (Lemma \ref{lemma_comparison}). Since $v$ has finite slope at $x_0=(0,0)$ and its value at $x_0$ coincides with $u$'s, we conclude that $u$ also has finite slope at $x_0$, as required.
\end{proof}

\begin{proof}[Proof of Theorem \ref{thm_B} \ref{item_thmB1}, case \eqref{eqn_finiteslopeproof1.1}] 
	We begin with the following observation: although the exponent $\alpha>2d-n$ is arbitrary, once the required statement is established for some exponent $\alpha_0$, it follows for all $\alpha\geq \alpha_0$. This is because condition \eqref{eqn_finiteslopeproof1.1} increases in strength as $\alpha$ increases:
	$$
	\det \D^2u\leq M(-u)^\alpha\ \Longrightarrow \ \det \D^2u\leq  M\|u\|_{L^\infty(\overline{\Omega})}^{\alpha-\alpha_0}(-u)^{\alpha_0}\ \text{ for }\alpha\geq \alpha_0
	$$
	(the constant $M$ varies, but this does not affect the statement). Therefore, we only need to treat those $\alpha$ that are arbitrarily close to $2d-n$. Since $2d-n<n$, we may henceforth assume that $\alpha<n$.
	
	Let $f$ and $v$ be the same as in the previous proof, except that we now:
	\begin{itemize}
		\item 
	take the parameter $M$ to be $1$ when applying Lemma \ref{lemma_barrier}, so that we have, in place of \eqref{eqn_finiteslopeproof1.2}, 
	\begin{equation}\label{eqn_finiteslopeproof4}
	\det \D^2f(y,z)\geq |z|^\alpha \ \ \text{ for all }(y,z)\in\Omega~;
	\end{equation}
	the constant $M$ in condition \eqref{eqn_finiteslopeproof1.1} is not used in this step and will appear only later;
		\item 
take the coefficient $C$ to be $C=2C_2/C_1$ in the definition \eqref{eqn_finiteslopeproof0} of $v$;
     \item 
     do not bound $u$ directly by $v$, but will find a constant multiple $\widetilde{v}$ of $v$ to bound $u$ instead. 
	\end{itemize}
By this choice of $C$ and \eqref{eqn_finiteslopeproof0}--\eqref{eqn_finiteslopeproof3}, we have
	$$
     v(y,z)\leq C_2|z|-C_1C|z|=-C_2|z|,\qquad 	v(y,z)\geq -C\xi\cdot z\geq -C|z|,
	$$
	and hence
	$-C^{-1}v(y,z)\leq |z|\leq -C_2^{-1}v(y,z)$
	for all $(y,z)\in\overline{\Omega}$. In view of \eqref{eqn_finiteslopeproof4}, it follows that
$$
	\det \D^2v(y,z)\geq |z|^\alpha\geq C_3(-v(y,z))^\alpha\ \ \text{ for all }(y,z)\in\Omega, \text{ where } C_3:=\min\big\{C^{-\alpha},\, C_2^{-\alpha}\big\}.
$$
Letting $\widetilde{v}=\lambda v$ with a constant $\lambda>0$, we deduce from the above inequality that
$$
\det\D^2\widetilde{v}=\lambda^n\det\D^2v\geq \lambda^nC_3(-v)^\alpha=\lambda^{n-\alpha}C_3(-\widetilde{v})^\alpha\  \text{ in }\Omega.
$$
Since we have assumed $\alpha<n$, we may adjust $\lambda$ to make the right-hand side no less than $M(-\widetilde{v})^\alpha$ (indeed, $\alpha\neq n$ suffices for this purpose). Putting everything together, we have
	$$
	u=0\geq \widetilde{v}\ \text{ on }\pa\Omega,\quad \det\D^2u\leq M(-u)^\alpha \ \text{ and }\ \det \D^2 \widetilde{v}\geq M(-\widetilde{v})^\alpha \ \text{ in }\Omega.
	$$	
Using Le's subcritical comparison principle (Lemma \ref{lemma_Lecomparison}), we deduce that $u\geq\widetilde{v}$ in $\Omega$, and conclude that $u$ has finite slope at $x_0$ as in the previous proof.
\end{proof}

When $\alpha<0$, the above proofs do not work if $u$ satisfies $\det\D^2u\leq M\,\dist(\,\cdot\,,\pa\Omega)^\alpha$ instead of \eqref{eqn_finiteslopeproof1} or \eqref{eqn_finiteslopeproof1.1}, because $\dist((y,z),\pa\Omega)^\alpha$ cannot be bounded from above by a constant times $\dist((y,z),L)=|z|^\alpha$. In fact, we have $\dist(\,\cdot\,,\pa\Omega)\leq \dist(\,\cdot\,,L)$ and hence $\dist(\,\cdot\,,\pa\Omega)^\alpha\geq \dist(\,\cdot\,,L)^\alpha$, but the reverse control up to a constant, which is essential for the proof, does not hold in general. Therefore, we may ask the following:

\begin{quote}
	\textbf{Question.} Suppose $d<n/2$ and $2d-n<\alpha<0$. Let $\Omega$ and $x_0\in\pa\Omega$ be as in Theorem \ref{thm_B}, and let $u\in C(\overline{\Omega})$ be a convex function satisfying
	$$
	\det \D^2u\leq M\,\dist(\,\cdot\,,\pa\Omega)^\alpha \  \text{ in }\Omega,\quad u|_{\pa\Omega}=0.
	$$
	Does $u$ have finite slope at $x_0$?
\end{quote}

\subsection{The existence result}\label{subsec_existence}
We proceed to prove Theorem \ref{thm_B} \ref{item_thmB2}. Namely, letting $\Gamma\subset\R^n$ be a convex cone with a $d$-dimensional ridge $L$, $\Omega\subset\Gamma$ be a bounded convex domain, and $g:\Omega\to\R_{\geq0}$ be a measurable function satisfying
\begin{equation}\label{eqn_proofexistence0.1}
	g(x)\leq M\,\dist(x,L)^\alpha\ \ \text{ for all }x\in\Omega
\end{equation}
with constants $M>0$ and $\alpha>2d-2n$, we shall show that there exists a  convex function $u\in C(\overline{\Omega})$ such that
\begin{equation}\label{eqn_proofexistence0}
\det\D^2u=g\  \text{ in }\Omega,\quad u|_{\pa\Omega}=0.
\end{equation}
The uniqueness of such a $u$ follows immediately from the comparison principle. Note that we do not actually need $L$ to contain some $x_0\in\pa\Omega$ as in the original statement of Theorem \ref{thm_B} \ref{item_thmB2}, although the statement is covered by the classical existence result (Lemma \ref{lemma_existence}) if this assumption fails.

The proof uses a standard exhaustion argument, where the barrier from Lemma \ref{lemma_barrier2} serves to bound the exhaustion sequence from below. However, this barrier has two issues: apart from the one mentioned after the statement of the lemma, the other is that it requires $\alpha< 2d-n$. The following lemma fixes both issues and formulates the barrier in a form suitable for use in the proof. 
%Recall that $B^m_r$ denotes the open ball of radius $r$ in $\R^m$ centered at the origin, and we omit the superscript if $m=n$.
\begin{lemma}\label{lemma_barrier3}
	Let $\Gamma\subset\R^n$ be a convex cone whose ridge $L$ has dimension $d$. Given $M,R>0$ and $\alpha>2d-2n$, there exists a convex function $v\in C(\Gamma\cap B_R)\cap C^\infty(\Gamma\cap B_R\setminus L)$ such that
%	\footnote{The set $\Gamma\cap B_R\setminus L\subset\R^n$ is not open, and $v\in C^\infty(\Gamma\cap B_R\setminus L)$ means $v$ admits a smooth extension to a neighborhood of each point of the set.} 
\begin{equation}\label{eqn_proofexistence0.2}
	\det\D^2v\geq M\,\dist(\,\cdot\,,L)^\alpha\ \text{ in }\Gamma\cap B_R\setminus L,\quad v|_{L\cap B_R}=0.
\end{equation}
\end{lemma}
\begin{proof}
	After a coordinate change, we may assume $\Gamma=\R^d\times\Gamma_0$, where $\Gamma_0\subset\R^{n-d}$ is a proper convex cone such that $\Gamma_0\setminus\{0\}$ is contained in $\R^{n-d}_+$. 
	We apply Lemma \ref{lemma_barrier2} with the same $R$ as here, but with $\alpha$ and $M$ replaced by new parameters $\alpha_0$ and $M_0$ such that
	$$
	2d-2n<\alpha_0<\min\big\{0,2d-n,\alpha\big\},\qquad
	M_0=MR^{\alpha-\alpha_0}.
	$$
	The lemma yields a convex function $f\in C(B^d_R\times\R^{n-d}_{\geq0})\cap C^\infty(B^d_R\times\R^{n-d}_+)$ with
	$$
	\det\D^2f(y,z)\geq M_0\,(z_1\cdots z_{n-d})^\frac{\alpha_0}{n-d}.
	$$
	Since $z_j\leq|z|$ ($j=1,\cdots,n-d$) and $\alpha_0,\alpha_0-\alpha<0$,  we have 
	$$
	M_0\,(z_1\cdots z_{n-d})^\frac{\alpha_0}{n-d}\geq M_0|z|^{\alpha_0}=M_0|z|^{\alpha_0-\alpha}|z|^\alpha\geq M_0 R^{\alpha_0-\alpha}|z|^\alpha=M|z|^\alpha
	$$
for all $z\in\R^{n-d}_+\cap B^{n-d}_R$.	Therefore, the restriction of $f$ to $\Gamma\cap B_R\subset B^d_R\times \big(\R^{n-d}_{\geq0}\cap B^{n-d}_R\big)$ gives the required function $v$.
%	Note that
%	\begin{equation}\label{eqn_proofexistence0.3}
%		C|z|\leq |z_j|\leq |z|\ \ \text{ for all }z=(z_1,\cdots,z_{n-d})\in\Gamma,\ j=1,\cdots,n-d,
%	\end{equation}
%	with a constant $C=C(\Gamma_0)\in (0,1)$, where the first ``$\leq$'' can be proved in the same way as Lemma \ref{lemma_twocones} and the second is trivial.	
%	When $\alpha>2d-n$, we obtain the required $v$ by restricting the function $f$ given by Lemma \ref{lemma_barrier} to $\Gamma\cap B_R$.  When $2d-2n<\alpha\leq 2d-n$, in view of \eqref{eqn_proofexistence0.3}, we may use the $f$ from Lemma \ref{lemma_barrier2} instead.
\end{proof}

With this lemma in hand, the rest of the proof is standard. We present it in detail for completeness.
\begin{proof}[Proof of Theorem \ref{thm_B} \ref{item_thmB2}]
Let $v$ be the function given by Lemma \ref{lemma_barrier3}, where we take the same $M$ as in \eqref{eqn_proofexistence0.1} and take $R$ large enough such that $\Omega\subset\Gamma\cap B_R$. Using the classical existence result (Lemma \ref{lemma_existence}), we obtain a sequence $(u_k)_{k=1,2,\cdots}$ of convex solutions to equation \eqref{eqn_proofexistence0} with truncated right-hand sides:
$$
\det\D^2u_k=\min\{g, k\}\ \text{ in }\Omega,\quad u_k|_{\pa\Omega}=0.
$$
By \eqref{eqn_proofexistence0.1} and \eqref{eqn_proofexistence0.2}, we have
$$
\det\D^2u_k\leq \det\D^2u_{k+1}\leq \det\D^2v\ \ \text{ in }\Omega,
$$
whereas $u_k|_{\pa\Omega}=u_{k+1}|_{\pa\Omega}=0\geq v|_{\pa\Omega}$, so the comparison principle (Lemma \ref{lemma_comparison}) gives
$$
u_k\geq u_{k+1}\geq v \ \text{ in }\Omega.
$$
Therefore, the pointwise limit function
$$
u:\Omega\to\R,\quad u(x):=\lim_{k\to\infty}u_k(x)
$$
is well defined. By Lemma \ref{lemma_uniform}, $u$ is also convex and the convergence $u_k\to u$ is uniform on compact subsets of $\Omega$. Then it follows from the stability of Monge-Amp\`ere measures (Lemma \ref{lemma_stability}) that 
$$
\det\D^2u=\lim_{k\to\infty}\det\D^2u_k=g.
$$ 

To show that $u$ provides a solution to \eqref{eqn_proofexistence0}, it only remains to check that
\begin{equation}\label{eqn_proofexistence1}
\liminf_{x\in\Omega,\ x\to x_0}u(x)=0
\end{equation}
for all $x_0\in\pa\Omega$, which is a necessary and sufficient condition for $u$ to be extendable to a continuous convex function on $\overline{\Omega}$ with $u|_{\pa\Omega}=0$ (see Remark \ref{rk_extension}). For $x_0\in F=\pa\Omega\cap L$, \eqref{eqn_proofexistence1} holds because
$$
0\geq \liminf_{x\in\Omega,\ x\to x_0}u(x)\geq \liminf_{x\in\Omega,\ x\to x_0}v(x)=v(x_0)=0.
$$
At $x_0\in \pa\Omega\setminus L$, although $v$ might not vanish, we can still deduce \eqref{eqn_proofexistence1} from the Alexandrov maximum principle. We now present a unified proof of \eqref{eqn_proofexistence1} for all $x_0\in\pa\Omega$.

Given $\epsilon>0$, consider the level sets
$$
E_{k,\epsilon}:=\big\{x\in\Omega\,\big|\,u_k\leq-\epsilon\big\},\quad E_\epsilon:=\big\{x\in\Omega\,\big|\,v\leq-\epsilon\big\}.
$$
Note that $u_k\geq v$ implies $E_{k,\epsilon}\subset E_\epsilon$.
The key to the proof is the fact that the mass of $E_{k,\epsilon}$ under the Monge-Amp\`ere measure of $u_k$ has an upper bound independent of $k$:
\begin{equation}\label{eqn_proofexistence2}
\mathcal{M}u_k(E_{k,\epsilon})\leq \mathcal{M}u_k(E_{\epsilon})\leq\mathcal{M}v(E_\epsilon)<+\infty,
\end{equation}
where the first ``$\leq$'' follows from $E_{k,\epsilon}\subset E_\epsilon$, the second from $\det\D^2 u_k\leq \det \D^2v$, and the last ``$<$'' holds  because $E_\epsilon$ is bounded away from the singular locus $L$ of $\det\D^2v$.

Since $u_k|_{\pa\Omega}=0$, $E_{k,\epsilon}$ is a compact convex set and we have $u_k|_{\pa E_{k,\epsilon}}=-\epsilon$. Therefore, we can apply the Alexandrov maximum principle (Lemma \ref{lemma_AlexandrovMP}) to the function $u_k+\epsilon$ on $E_{k,\epsilon}$. In view of \eqref{eqn_proofexistence2}, this gives
$$
u_k(x)\geq-\epsilon-C\, \dist(x,\pa E_{k,\epsilon})^\frac{1}{n}\ \ \text{ for all }x\in E_{k,\epsilon},
$$
where the constant $C=C(n,\diam(\Omega),\mathcal{M}v(E_\epsilon))>0$  depends on $\epsilon$ but not on $k$. Since $\dist(x,\pa E_{k,\epsilon})< \dist(x,\pa\Omega)$ for $x\in E_{k,\epsilon}$ and $u_k(x)>-\epsilon$ for $x\in \Omega\setminus E_{k,\epsilon}$, it follows that
$$
u_k(x)\geq -\epsilon-C\,\dist(x,\pa\Omega)^\frac{1}{n}\ \ \text{ for all }x\in\Omega.
$$
The same holds for $u(x)$ by letting $k\to\infty$. Therefore, for any $x_0\in\pa\Omega$, we have
$$
0\geq \liminf_{x\in\Omega,\,x\to x_0}u(x)\geq -\epsilon -C\liminf_{x\in\Omega,\,x\to x_0}\dist(x,\pa\Omega)^\frac{1}{n}=-\epsilon.
$$
Since $\epsilon>0$ is arbitrary, this implies \eqref{eqn_proofexistence1} and completes the proof.
\end{proof}

%\section{Open questions and future directions}\label{sec_open}
%
%\subsection*{Finite-slope result for $\det\D^2u=\dist(\,\cdot\,,\pa\Omega)^\alpha$}
%
%\subsection*{Quantitative version}
%
%\subsection*{Generalization under the Jian-Wang condition}
%
%\subsection*{More general equations}

\bibliographystyle{plain} \bibliography{infiniteness}
\end{document}